\documentclass[12pt]{article} 

\usepackage{amssymb}
\usepackage{amsthm}
\usepackage{mathrsfs}
\usepackage{mathtools}
\usepackage[pdftex]{graphicx}
\usepackage{caption}
\usepackage{subcaption}
\usepackage{placeins}
\usepackage{url}
\usepackage{tikz} 
\usetikzlibrary{matrix,arrows}
\usepackage{dsfont}
\usepackage{fullpage}
\usepackage{hyperref}

\theoremstyle{definition}
\newtheorem{definition}{Definition}
\newtheorem*{open question}{Open Question}
\newtheorem{remark}{Remark}
\theoremstyle{plain}
\newtheorem{lemma}{Lemma}
\newtheorem{theorem}{Theorem}
\newtheorem*{theorem 7}{Theorem~7 of \cite{Coordinate-free}, reformulated}
\newtheorem*{The Evasion Problem}{The Evasion Problem}

\newcommand{\I}{\ensuremath{\mathbb{I}}}

\newcommand{\R}{\ensuremath{\mathbb{R}}}
\newcommand{\1}{\ensuremath{\mathds{1}}}

\newcommand{\xD}{\ensuremath{\mathcal{D}}}

\DeclareMathOperator{\Hom}{Hom}

\DeclareMathOperator{\VR}{VR}

\begin{document}

\title{Evasion Paths in Mobile Sensor Networks}
\author{Henry Adams and Gunnar Carlsson}
\maketitle

\begin{abstract}
Suppose that ball-shaped sensors wander in a bounded domain. A sensor doesn't know its location but does know when it overlaps a nearby sensor. We say that an evasion path exists in this sensor network if a moving intruder can avoid detection. In {\em Coordinate-free coverage in sensor networks with controlled boundaries via homology}, Vin de~Silva and Robert Ghrist give a necessary condition, depending only on the time-varying connectivity data of the sensors, for an evasion path to exist. Using zigzag persistent homology, we provide an equivalent condition that moreover can be computed in a streaming fashion. However, no method with time-varying connectivity data as input can give necessary and sufficient conditions for the existence of an evasion path. Indeed, we show that the existence of an evasion path depends not only on the fibrewise homotopy type of the region covered by sensors but also on its embedding in spacetime. For planar sensors that also measure weak rotation and distance information, we provide necessary and sufficient conditions for the existence of an evasion path.
\end{abstract}

\section{Introduction}

In minimal sensor network problems, one is given only local data measured by many weak sensors but tries to answer a global question \cite{EstrinCullerPisterSukhatme, GaoGuibas}. Tools from topology can be useful for this passage from local to global. For example, \cite{TargetEnumeration} combines redundant local counts of targets to obtain an accurate global count using integration with respect to Euler characteristic. Coverage problems are another class of problems in minimal sensing: when sensors are scattered throughout a domain, can we determine if the entire domain is covered? See \cite{AhmedKanhereSalilJha, FanJin, GhoshDas, Wang} for surveys of coverage problems, and see \cite{Coordinate-free, CoverageViaPersistentHomology} for topological approaches.

We are interested in the following mobile sensor network coverage problem from \cite{Coordinate-free}. Suppose that ball-shaped sensors wander in a bounded domain. A sensor can't measure its location but does know when it overlaps a nearby sensor. We say that an evasion path exists if a moving intruder can avoid being detected by the sensors. Can we determine if an evasion path exists? We refer to this question as the evasion problem. The evasion problem can also be described as a pursuit-evasion problem in which the domain is continuous and bounded, there are multiple sensors searching for intruders, and an intruder moves continuously and with arbitrary speed. We do not control the motions of the sensors; the sensors wander continuously but otherwise arbitrarily. We cannot measure the locations of the sensors but instead know only their time-varying connectivity data. Using this information, we would like to determine whether it is possible for an intruder to avoid the sensors. See \cite{ChungHollingerIsler} for a survey of related pursuit-evasion problems, and see \cite{MobilityImprovesCoverage, DetectionOfIntelligentMobile} for scenarios in which the motion of the sensors or intruders can be controlled.

After introducing the evasion problem in \cite{Coordinate-free}, de~Silva and Ghrist give a necessary homological condition for an evasion path to exist. Using zigzag persistent homology, we provide an equivalent condition that moreover can be computed in a streaming fashion. However, it turns out that homology alone is not sufficient for the evasion problem. Indeed, neither the fibrewise homotopy type of the sensor network nor any invariants thereof determine if an evasion path exists; we show that the fibrewise embedding of the sensor network into spacetime also matters. Knowing this, we provide necessary and sufficient conditions for the existence of an evasion path for planar sensors that can also measure weak rotation and distance data.

In Section~\ref{S: Fibrewise Spaces} we provide background material on fibrewise spaces. We define the evasion problem in Section~\ref{S: The Evasion Problem}, and in Section~\ref{S: Work of de Silva and Ghrist} we describe the work of de~Silva and Ghrist. We introduce zigzag persistence in Section~\ref{S: Zigzag Persistence} and apply it to the evasion problem in Section~\ref{S: Applying Zigzag Persistence to the Evasion Problem}. However, zigzag persistence does not give a complete solution to the evasion problem. Indeed, in Section~\ref{S: Dependence on the Embedding} we show the existence of an evasion path depends not only on the fibrewise homotopy type of the sensor network but also on the ambient isotopy class of its embedding in spacetime. In Section~\ref{S: Sensors Measuring Cyclic Orderings} we restrict attention to planar sensors measuring cyclic orderings and provide an if-and-only-if result. We conclude in Section~\ref{C: Conclusions} and describe possible directions for future work.

\section{Fibrewise Spaces}\label{S: Fibrewise Spaces}

Since we are studying mobile sensors, both the region covered by the sensors and the uncovered region change with time. In this section we encode the notion of time-varying spaces using the language of a fibrewise spaces. We also consider fibrewise maps between fibrewise spaces, what it means for two fibrewise maps to be fibrewise homotopic, and what it means for two fibrewise spaces to be fibrewise homotopy equivalent. See \cite{FibrewiseHomotopyTheory} for more information on fibrewise homotopy theory.

A fibrewise space is a space equipped with a notion of time. More precisely, let $I = [0,1]$ be the closed unit interval. A fibrewise space is a topological space $Y$ equipped with a continuous map $p \colon Y \to I$ to time; see Figure~\ref{F: cartoonYes} for an example. For any point $y \in Y$ one can think of $p(y)$ as the time coordinate associated to this point. Given two fibrewise spaces $p \colon Y \to I$ and $p' \colon W \to I$, a continuous map $f \colon Y \to W$ is said to be fibrewise if $p' \circ f = p$. In other words, a fibrewise map is time-preserving. A section for a fibrewise space $p \colon Y \to I$ is a fibrewise map $s \colon I \to Y$, that is, a continuous map $s \colon I \to Y$ with $p(s(t)) = t$ for all $t \in I$.

\begin{figure}[h]
\begin{center}
    	\includegraphics[width=3in]{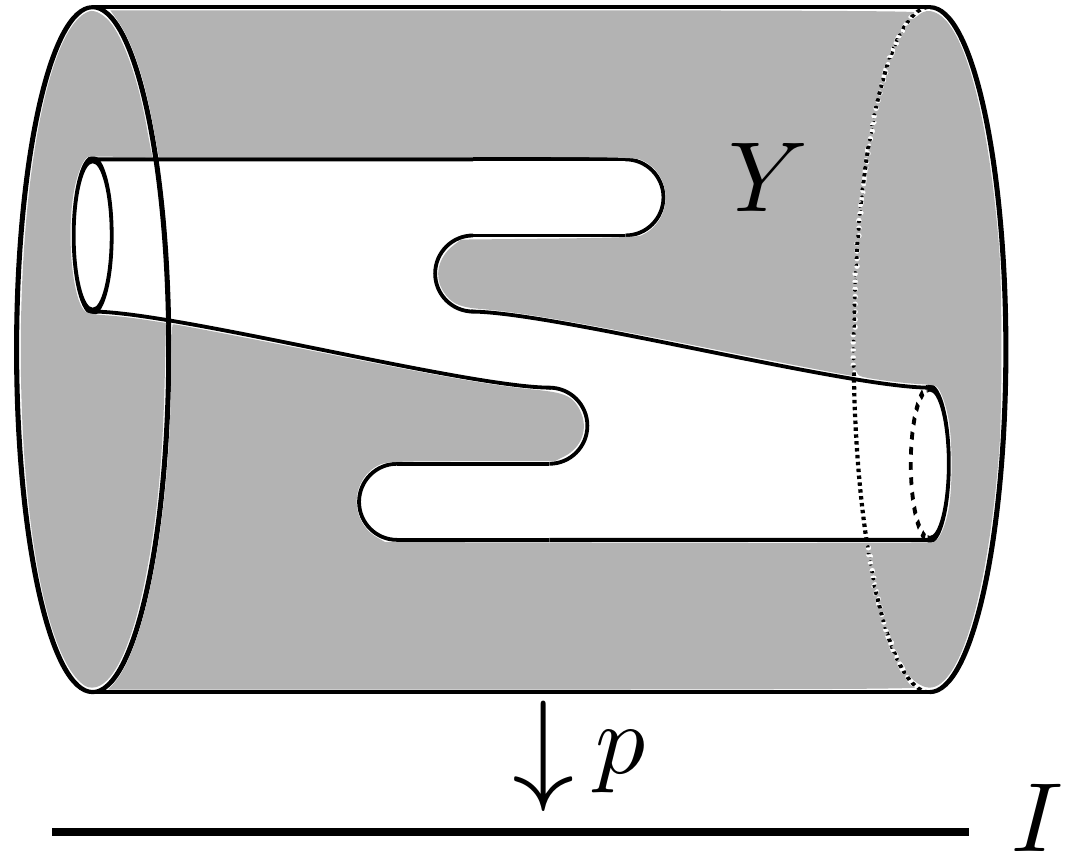}
\end{center}
\caption{A fibrewise space $p \colon Y \to I$.}
\label{F: cartoonYes}
\end{figure}

Roughly speaking, two fibrewise maps are fibrewise homotopic when one can be deformed to the other in a time-preserving manner. More precisely, two fibrewise maps $f_0, f_1 \colon Y \to W$ are fibrewise homotopic if there is a homotopy $F \colon Y \times I \to W$  with $F(\quad, 0) = f_0$, with $F(\quad, 1) = f_1$, and with each $F(\quad, t)$ a fibrewise map. The homotopy $F$ gives a continuous and time-preserving deformation from $f_0$ to $f_1$. Two fibrewise spaces are fibrewise homotopy equivalent when they are homotopy equivalent in a time-preserving manner. More explicitly, a fibrewise map $f \colon Y \to W$ is a fibrewise homotopy equivalence if there is a fibrewise map $f' \colon W \to Y$ with compositions $f' \circ f$ and $f \circ f'$ each fibrewise homotopic to the corresponding identity map; in this case we say that fibrewise spaces $Y$ and $W$ are fibrewise homotopy equivalent.
\FloatBarrier

\section{The Evasion Problem}\label{S: The Evasion Problem}

The evasion problem we consider is introduced in \cite{Coordinate-free}, and we present it here with a few minor changes. Let $\xD \subset \R^d$ be a bounded domain homeomorphic to a $d$-dimensional ball, where $d \geq 2$. Suppose a finite set $S$ of sensor nodes moves inside this domain over the time interval $I = [0, 1]$, with each sensor $v \in S$ a continuous path $v \colon I \to \xD$. We assume that two distinct sensors never occupy the same location: for sensors $v \neq \tilde{v}$ we have $v(t) \neq \tilde{v}(t)$ for all $t$. Let $B_{v(t)} = \{y \in \xD\ |\ \|v(t) - y\| \leq 1\}$ be the unit ball covered by sensor $v$ at time $t$. The sensors can't measure their locations but two sensors do know when they overlap. This allows us to measure the time-varying connectivity graph of the sensors. The connectivity graph at time $t$ has the set of sensors $S$ as its vertex set and has an edge between sensors $v$ and $\tilde{v}$ when $B_{v(t)} \cap B_{\tilde{v}(t)} \neq \emptyset$; see Figure~\ref{F: sensorBalls_balls_connectivityGraph_Cech}(b). We assume there is an immobile subset of fence sensors whose union of balls contains the boundary $\partial \xD$ and is homotopy equivalent to $\partial \xD$.

\begin{figure}[h]
\begin{center}
	\begin{subfigure}[t]{2.12in}
		\includegraphics[width=2.12in]{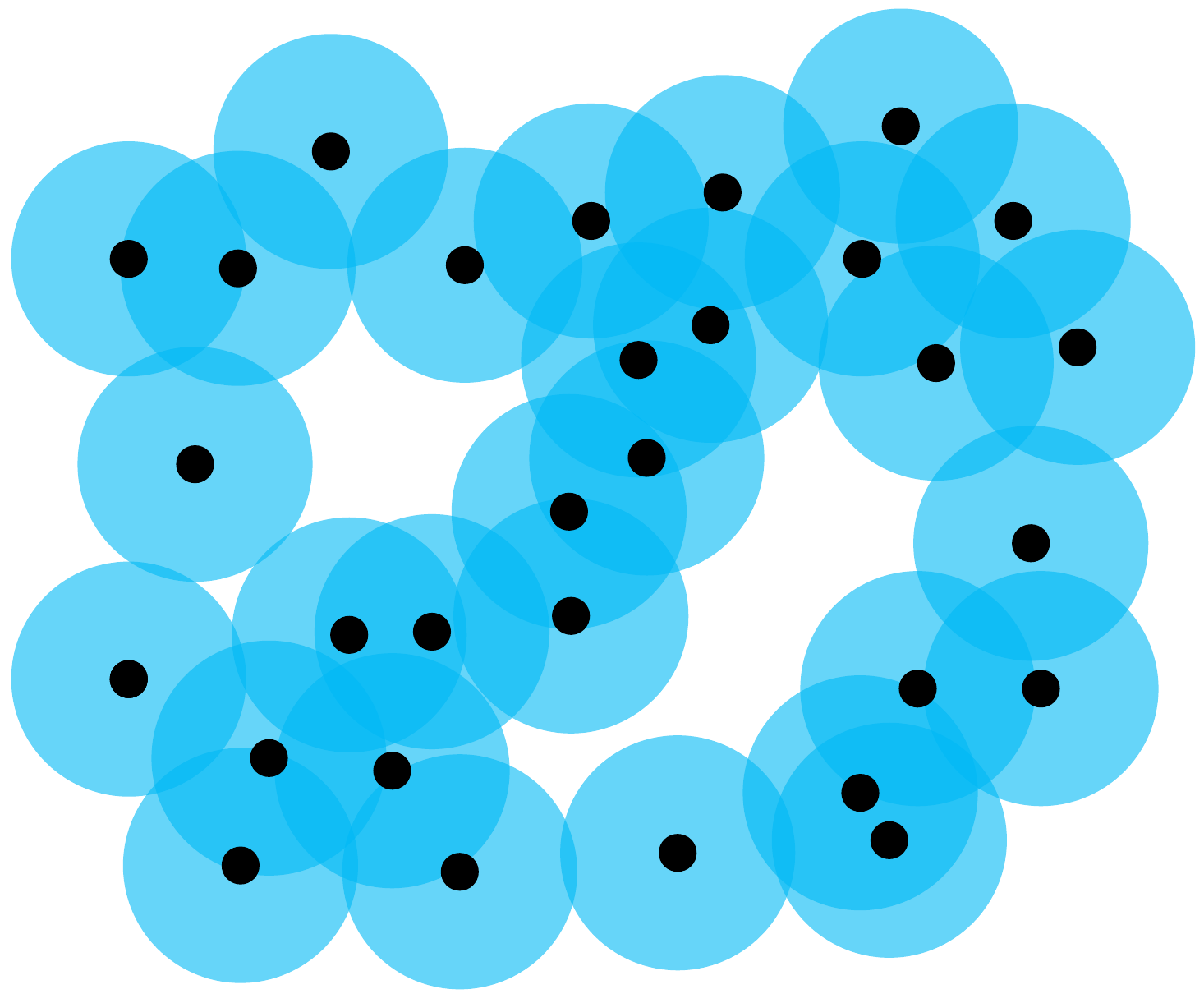}
		\caption{}
	\end{subfigure}
	\begin{subfigure}[t]{2.12in}
		\includegraphics[width=2.12in]{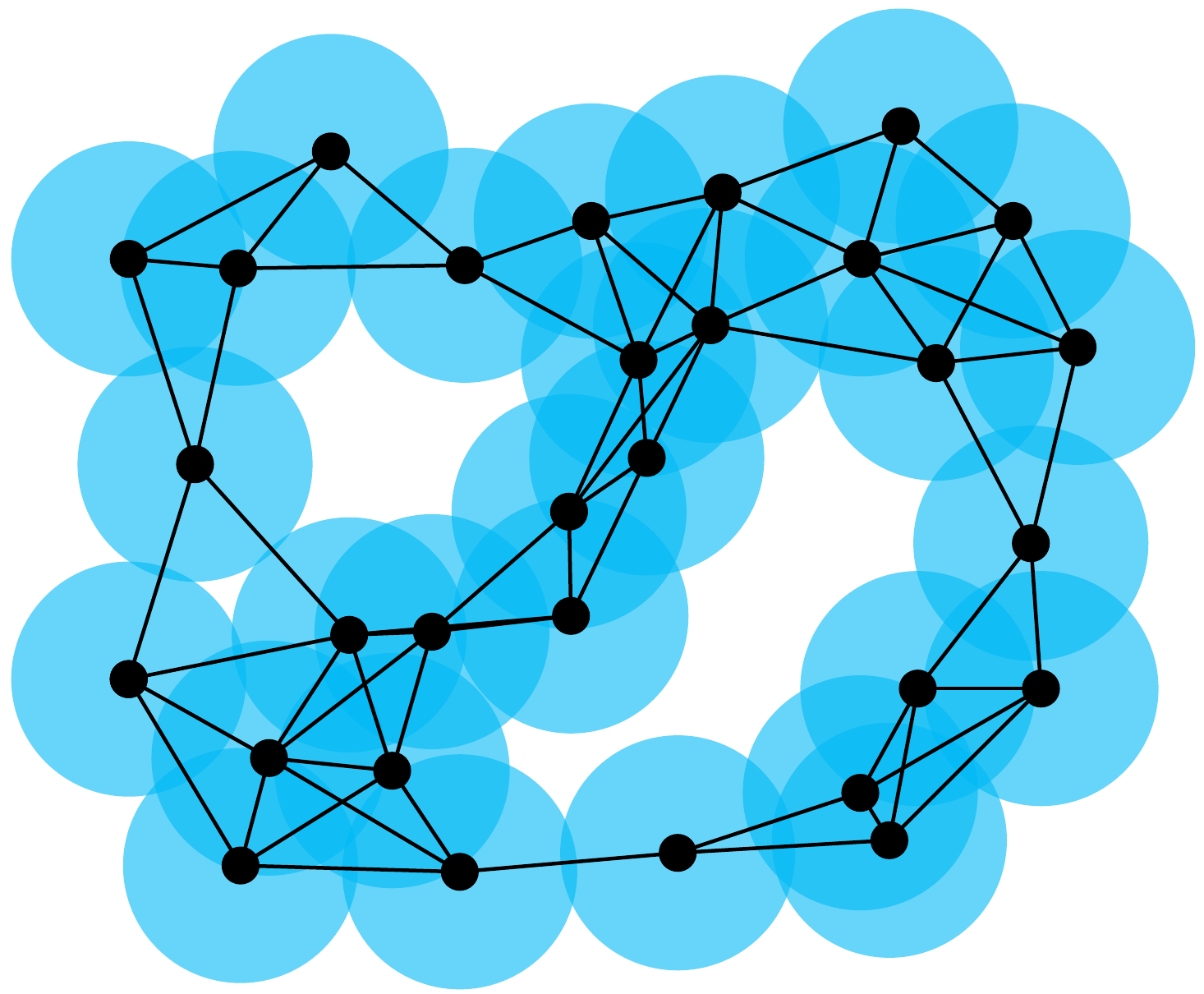}
		\caption{}
	\end{subfigure}
	\begin{subfigure}[t]{2.12in}
		\includegraphics[width=2.12in]{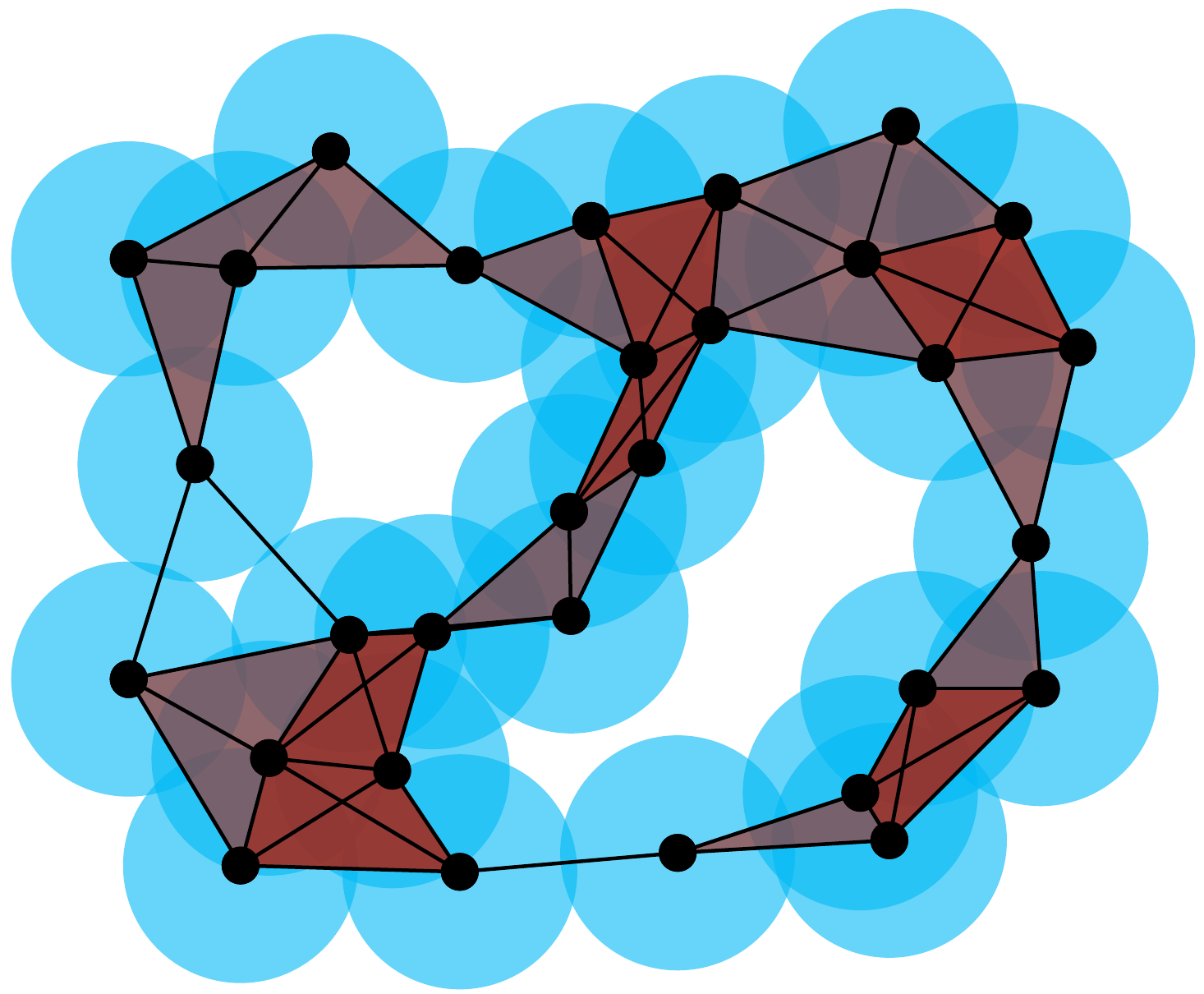}
		\caption{}
	\end{subfigure}
\end{center}
\caption{(a) A sensor network at a fixed point in time, (b) its connectivity graph, and (c) its \u Cech complex. See \href{http://www.ima.umn.edu/~henrya/research/DependenceOnTheEmbedding.avi}{Extension~1} for a video of mobile sensor networks.}
\label{F: sensorBalls_balls_connectivityGraph_Cech}
\end{figure}

The union of the balls $X(t) = \cup_{v \in S}B_{v(t)}$ is the region covered by the sensors at time $t$, and its complement $X(t)^c = \xD \setminus X(t)$ is the uncovered region at time $t$. Let
$$X = \cup_{t \in I}X(t) \times \{t\} \subset \xD \times I$$
be the subset of spacetime covered by sensors, and let $X^c = (\xD \times I) \setminus X$ be the uncovered region in spacetime. Both $X$ and $X^c$ are fibrewise spaces, that is, spaces equipped with projection maps $X \to I$ and $X^c \to I$ to time. See Figure~\ref{F: cartoonYesTeleport} for two examples.

\begin{figure}[h]
\begin{center}
	\begin{subfigure}[t]{2.5in}
		\includegraphics[width=2.5in]{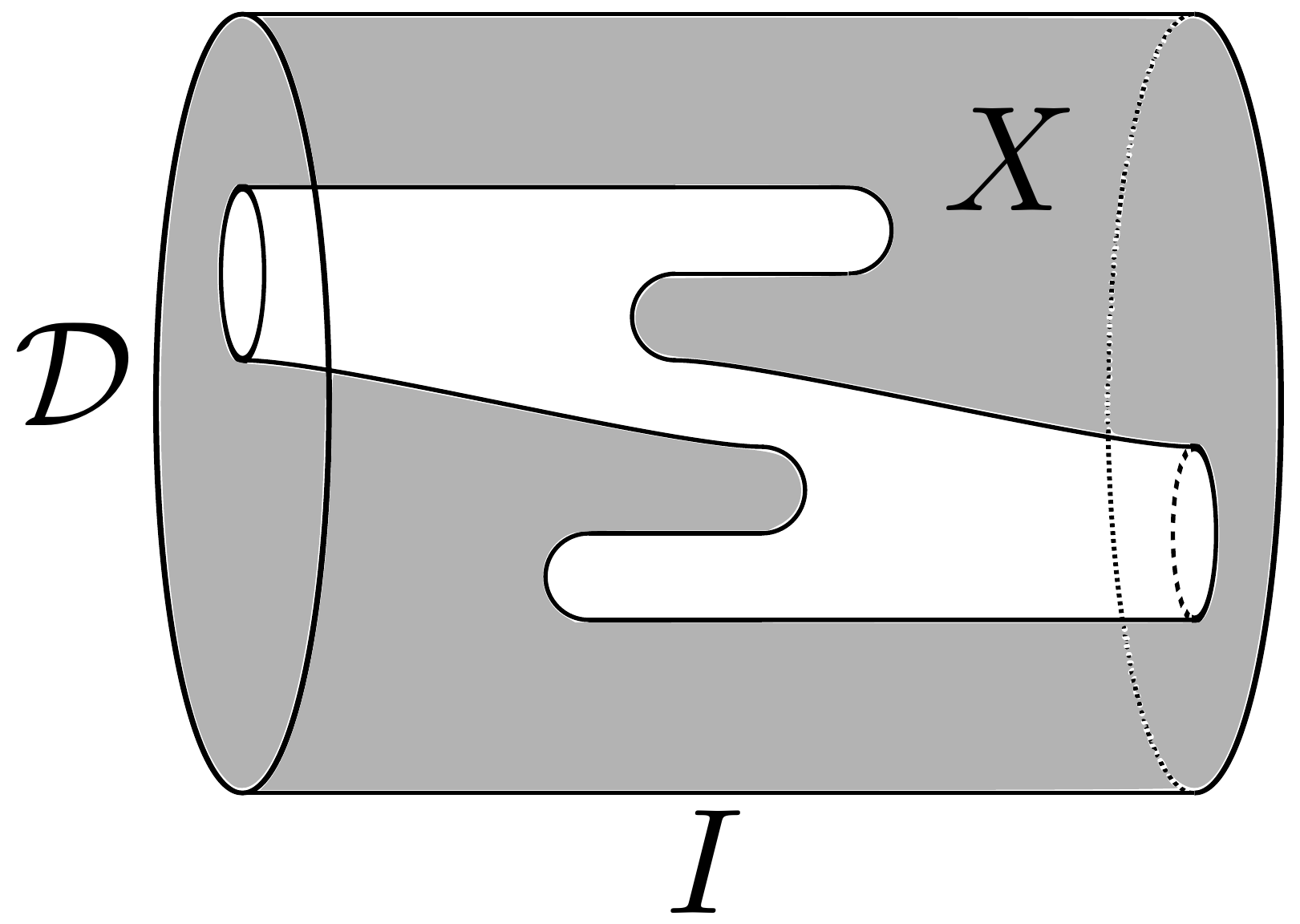}
	\end{subfigure}
	\hspace{5mm}
	\begin{subfigure}[t]{2.5in}
		\includegraphics[width=2.5in]{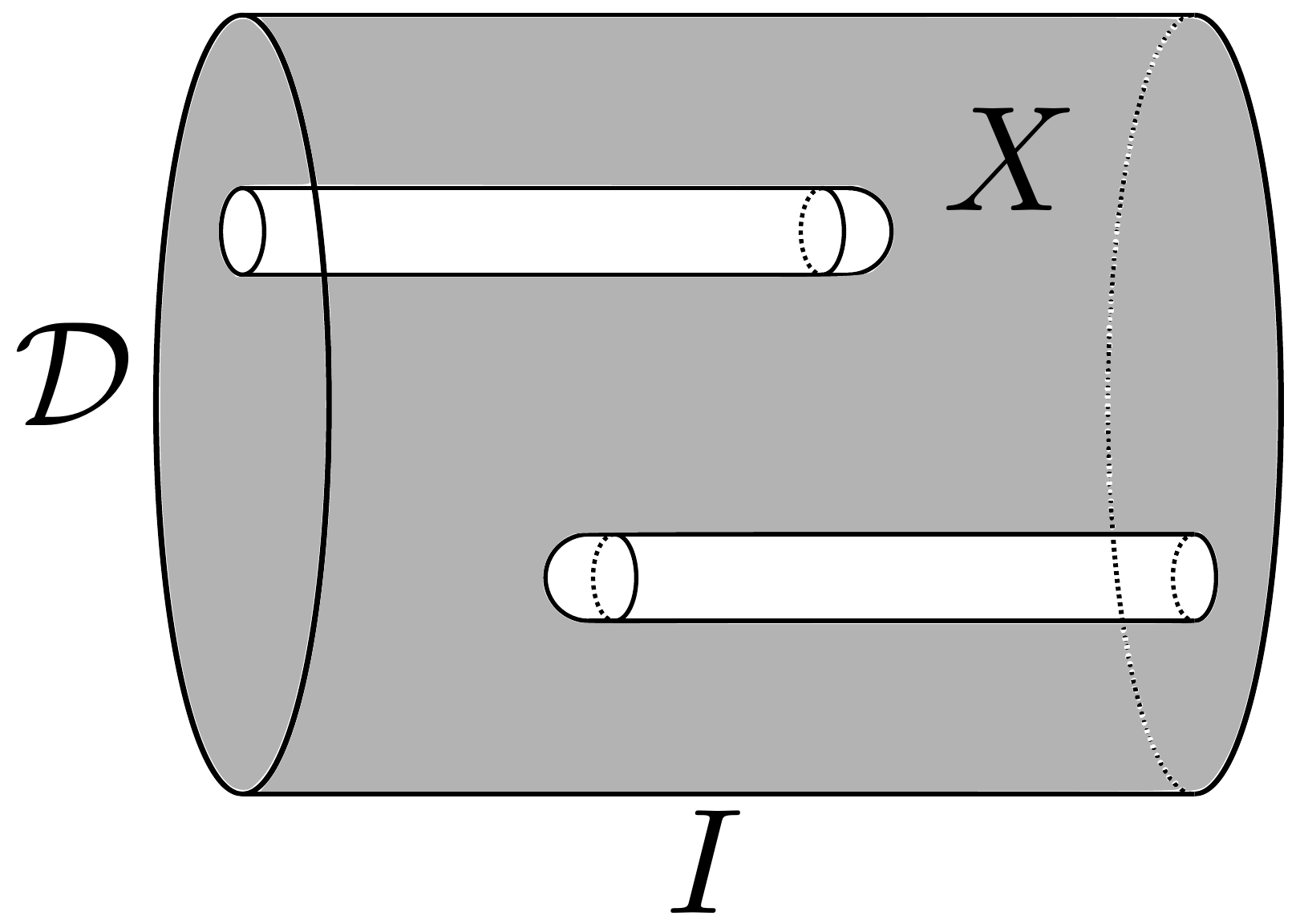}
	\end{subfigure}
\end{center}
\caption{We have drawn two planar sensor networks with domain $\xD \subset \R^2$ on the vertical axis and with time $I$ on the horizontal axis. The region $X$ in spacetime covered by the sensors is drawn in gray, and the uncovered region $X^c$ is drawn in white. The network on the left contains an evasion path. The network on the right does not contain an evasion path because an intruder must move continuously and cannot teleport locations.}
\label{F: cartoonYesTeleport}
\end{figure}

Potentially there are also intruders moving continuously in this domain. The intruders would like to avoid being detected by the sensors, but an intruder is detected at time $t$ if it lies in the covered region $X(t)$. We say that an evasion path exists when it is possible for a moving intruder to avoid being seen by the sensors.

\begin{definition}
An {\em evasion path} in a sensor network is a section $s \colon I \to X^c$ of the projection map $p \colon X^c \to I$. Equivalently, an evasion path is a continuous map $s \colon I \to \xD$ such that $s(t) \notin X(t)$ for all $t$.
\end{definition}

Given a sensor network, we would like to determine whether or not an evasion path exists. However, connectivity graphs alone cannot determine the existence of an evasion path. Consider the two sensor networks in Figure~\ref{F: connectivityGraphInsufficient}, and suppose that in each case all three sensors are immobile over the entire time interval. Then the two networks have the same connectivity graphs at each point in time, but the sensor network on the left contains an evasion path while the network on the right does not.

\begin{figure}[h]
\begin{center}
	\begin{subfigure}[t]{1.7in}
		\includegraphics[width=1.7in]{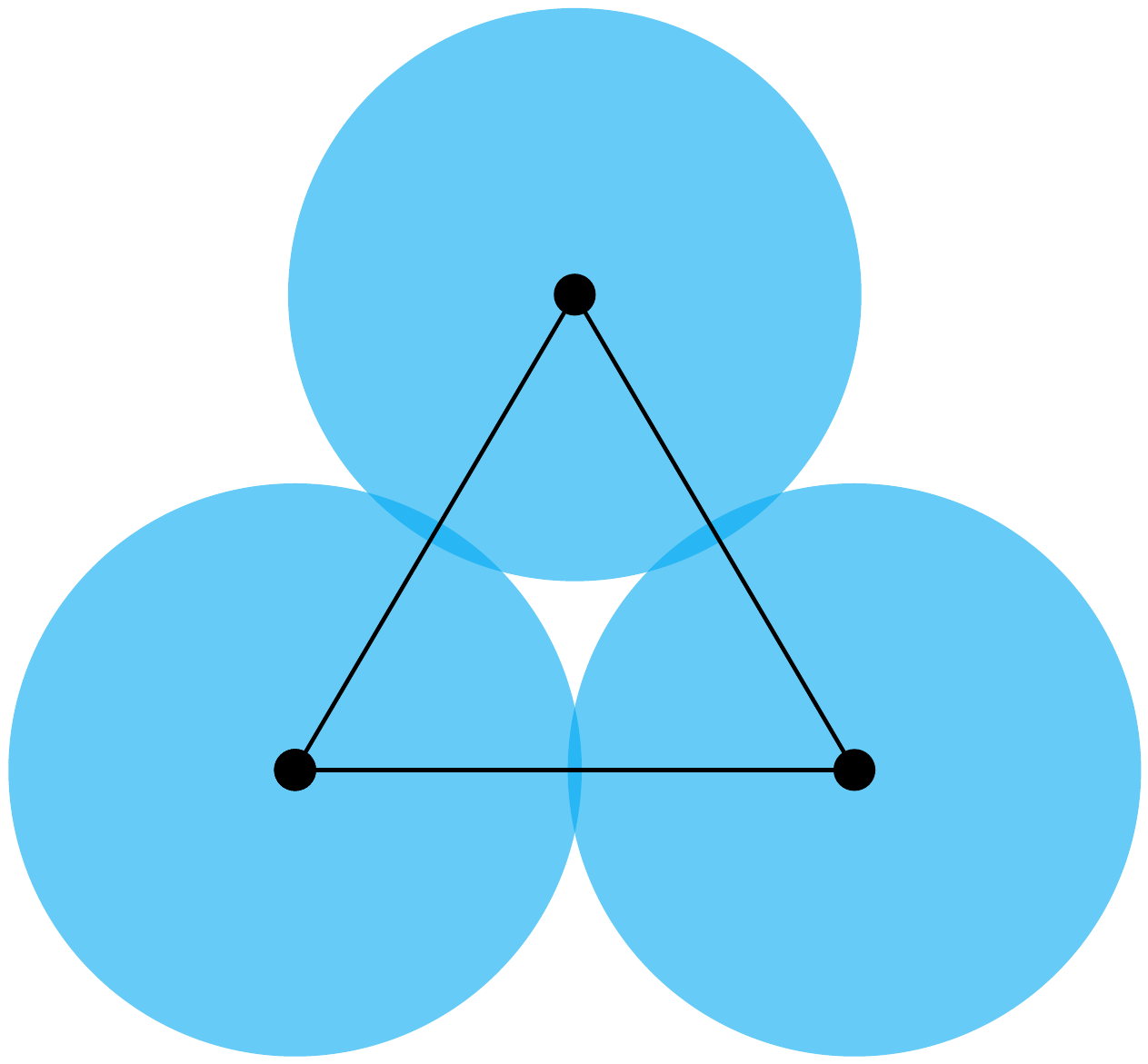}
		\caption{}
	\end{subfigure}
	\hspace{10mm}
	\begin{subfigure}[t]{1.7in}
		\includegraphics[width=1.7in]{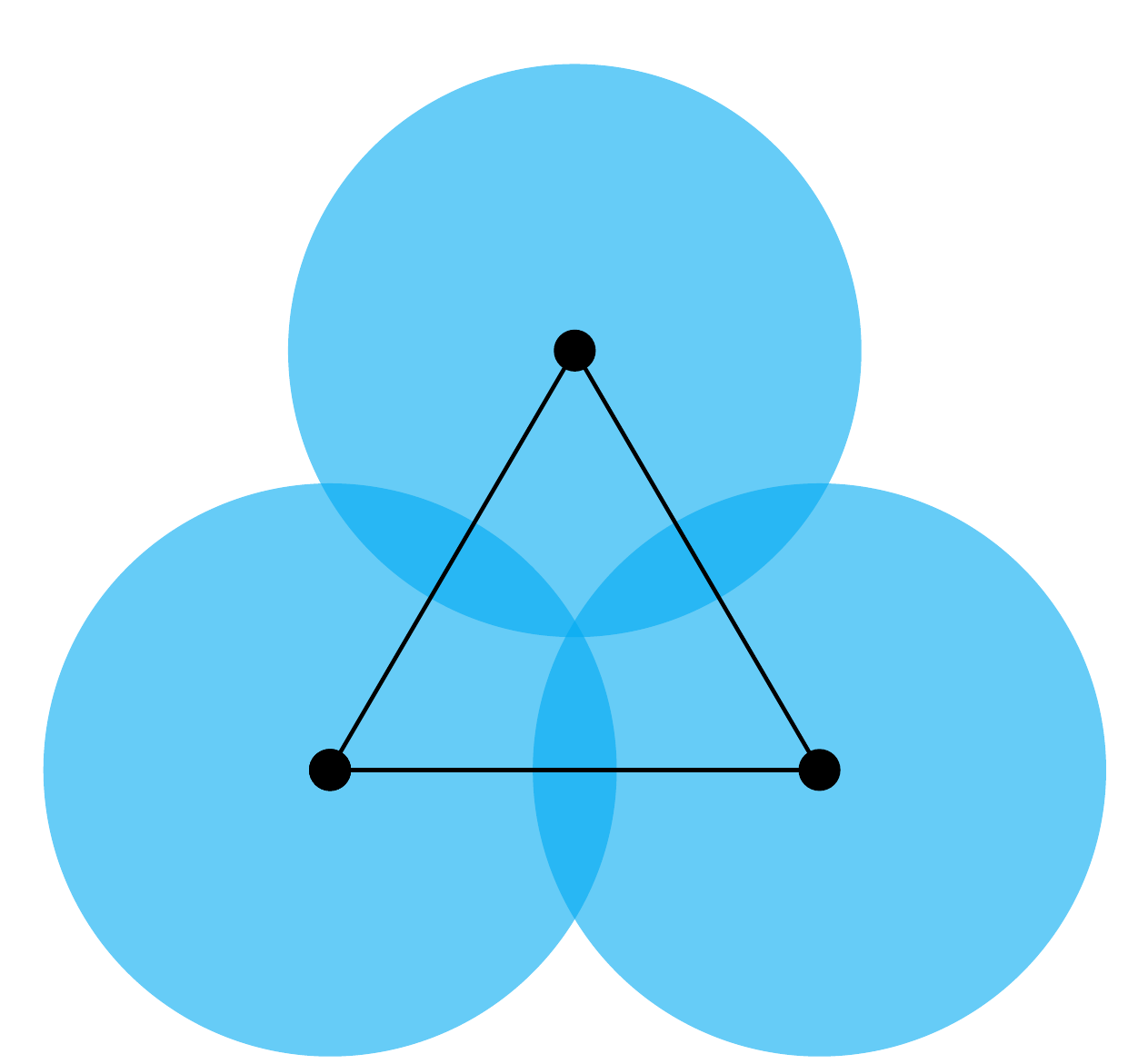}
		\caption{}
	\end{subfigure}
\end{center}
\caption{Let (a) and (b) be two different sensor networks. Imagine that no sensor moves over the entire time interval. Then network (a) has an evasion path while network (b) does not, even though the two networks have the same connectivity graph at each point in time.}
\label{F: connectivityGraphInsufficient}
\end{figure}

Since connectivity graphs alone are insufficient, we consider \u Cech simplicial complexes encoding higher connectivity information. The \u Cech simplicial complex $C(t)$ of the sensors is the nerve of the unit balls $\{B_{v(t)}\}_{v \in S}$ \cite{ComputationalTopology}. This means that the vertex set of $C(t)$ is the set of sensors $S$, and we have a $k$-simplex when the intersection of the corresponding $(k+1)$-balls is nonempty. 
That is, simplex $[v_0v_1 \ldots v_k]$ is in $C(t)$ when 
$$B_{v_0(t)} \cap B_{v_1(t)} \cap \ldots \cap B_{v_k(t)} \neq \emptyset.$$
See Figure~\ref{F: sensorBalls_balls_connectivityGraph_Cech}(c) for an example. Note that the 1-skeleton of $C(t)$ is the connectivity graph at time $t$. An important property is that the \u Cech complex $C(t)$ is homotopy equivalent to the union of the balls $X(t)$ by the nerve lemma \cite[Corollary~4G.3]{Hatcher}. We are now ready to state the evasion problem.

\begin{The Evasion Problem}
Given the time-varying \u Cech complex $C(t)$ of a sensor network over all times $t \in I$, can one determine if an evasion path exists?
\end{The Evasion Problem}

\begin{remark}
One can recover the fibrewise homotopy type of the covered region $X$ of a sensor network from the time-varying \u Cech complex.
\end{remark}

\begin{remark} In the static setting in which the sensors do not move, one can use the \u Cech complex to determine whether or not the sensors cover the entire domain $\xD$. The evasion problem asks whether an analogous statement is true in the setting of mobile sensors.
\end{remark}

In applications it is generally unreasonable to assume that our sensors can measure \u Cech complexes. This would require the task of detecting $k$-fold intersections, which is not possible under many models of minimal sensing. However, we can approximate the \u Cech complex from either above or below using the Vietoris--Rips complex \cite{Vietoris}. The Vietoris--Rips complex is the maximal simplicial complex built on top of the connectivity graph, and hence can be recovered by sensors measuring only two-fold overlaps. This approximation allows us to take results based on \u Cech complexes and produce analogous approximate results using only Vietoris--Rips complexes; see Appendix~\ref{A: Cech Complex Approximations} for more details. For example, the results in \cite{Coordinate-free} are stated in terms of Vietoris--Rips complexes. We avoid such approximations and instead use \u Cech complexes.

\section{Work of de~Silva and Ghrist}\label{S: Work of de Silva and Ghrist}

In order to explain de~Silva and Ghrist's work on the evasion problem, we first define the stacked \u Cech complex. The stacked \u Cech complex is a single cell complex encoding the \u Cech simplicial complexes $C(t)$ for all times $t \in I$. We assume there are only a finite number of times
$$0 < t_1 < \ldots < t_n < 1$$
when the \u Cech complex changes. Hence for $t$ and $t'$ in either $(t_i, t_{i+1})$, $[0, t_1)$, or $(t_n, 1]$, we have $C(t) = C(t')$. Moreover, we assume that at each time $t_i$ simplices are either added to or removed from the \u Cech complex but not both. Since the sensors balls are closed, a simplex $\sigma$ is
\begin{itemize}
\item added at time $t_i$ if $\sigma \in C(t_i)$ but $\sigma \notin C(t)$ for $t \in (t_{i-1}, t_i)$, and 
\item removed at time $t_i$ if $\sigma \in C(t_i)$ but $\sigma \notin C(t)$ for $t \in (t_i, t_{i+1})$.
\end{itemize}
Choose interleaving times
$$0 = s_0 < t_1 < s_1 < \ldots < t_n < s_n = 1.$$

\begin{definition}
The {\em stacked \u Cech complex} $p \colon SC \to I$ is the fibrewise space obtained from the disjoint union $\amalg_{i=0}^n C(s_i) \times [t_i, t_{i+1}]$, where $t_0 = 0$ and $t_{n+1} = 1$, by identifying
\begin{itemize}
\item $C(s_{i-1}) \times \{t_i\}$ as a subset of $C(s_i) \times \{t_i\}$ if simplices are added at $t_i$, and
\item $C(s_i) \times \{t_i\}$ as a subset of $C(s_{i-1}) \times \{t_i\}$ if simplices are removed at $t_i$.
\end{itemize}
Map $p \colon SC \to I$ is the projection onto the second coordinate, and note that $p^{-1}(t) = C(t)$. 
\end{definition}

This definition is similar to the definition of the stacked Vietoris--Rips complex in \cite{Coordinate-free}. See Figure~\ref{F: stackedCech} for a small example.

\begin{figure}[h]
\begin{center}
	\begin{subfigure}[t]{4in}
		\includegraphics[width=4in]{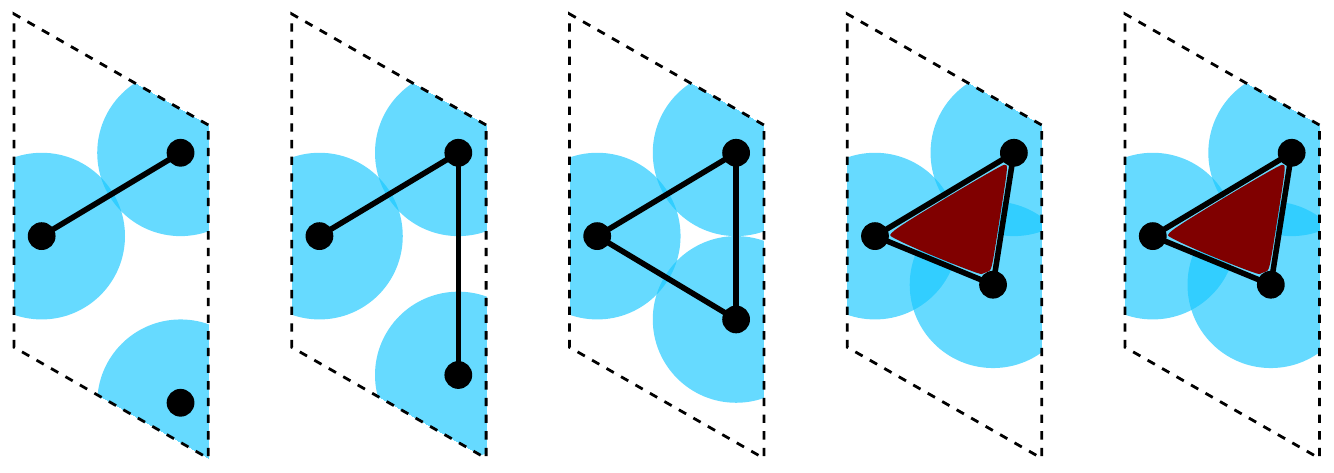}
	\end{subfigure}
	\begin{subfigure}[t]{4in}
		\includegraphics[width=4in]{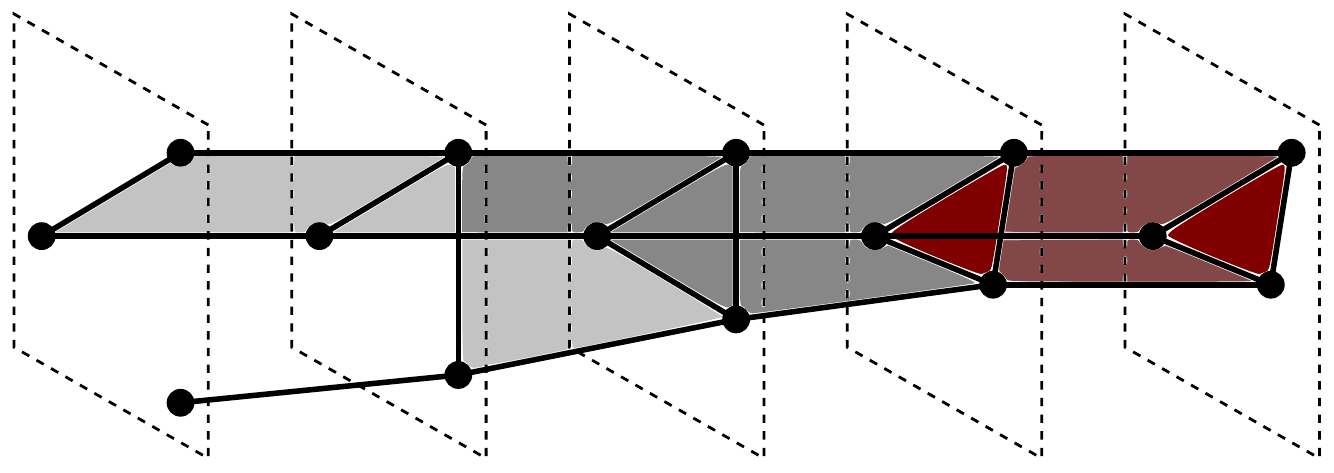}
	\end{subfigure}
\end{center}
\caption{The stacked \u Cech complex for three sensor nodes. The top row shows how the \u Cech complex changes: initially the \u Cech complex consists of an edge and a vertex, and as the sensors move closer together two more edges and a 2-simplex are added. The bottom row shows the stacked \u Cech complex, obtained by adding prism cells.}
\label{F: stackedCech}
\end{figure}

De Silva and Ghrist give a partial answer to the evasion problem in Theorem~7 of \cite{Coordinate-free}. We state their result using the stacked \u Cech complex instead of the stacked Vietoris--Rips complex, and for $\xD \subset \R^d$ with dimension $d$ arbitrary. Recall that a subset of immobile fence sensors covers the boundary $\partial \xD$, and let $F \times [0, 1]$ be the subcomplex of the stacked \u Cech complex $SC$ consisting of only these fence sensors.

\begin{theorem 7}
If there exists some $[\alpha] \in H_d(SC, F \times [0,1])$ with $0 \neq [\partial \alpha] \in H_{d-1}(F \times [0,1])$, then there is no evasion path in the sensor network.
\end{theorem 7}

\begin{figure}[h]
\begin{center}
	\begin{subfigure}[t]{2.2in}
		\includegraphics[width=2.2in]{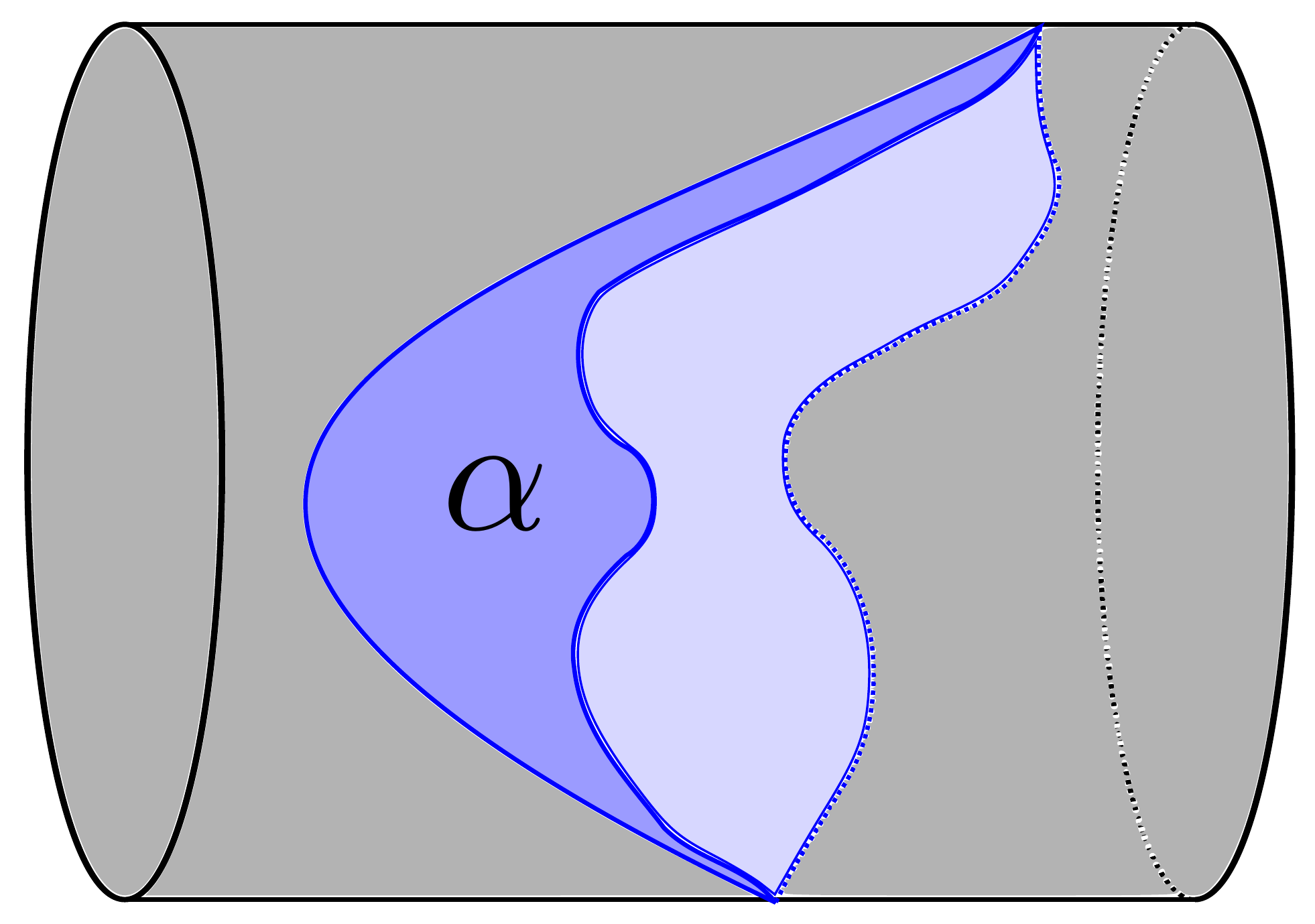}
		\caption{}
	\end{subfigure}
	\hspace{10mm}
	\begin{subfigure}[t]{2.2in}
		\includegraphics[width=2.2in]{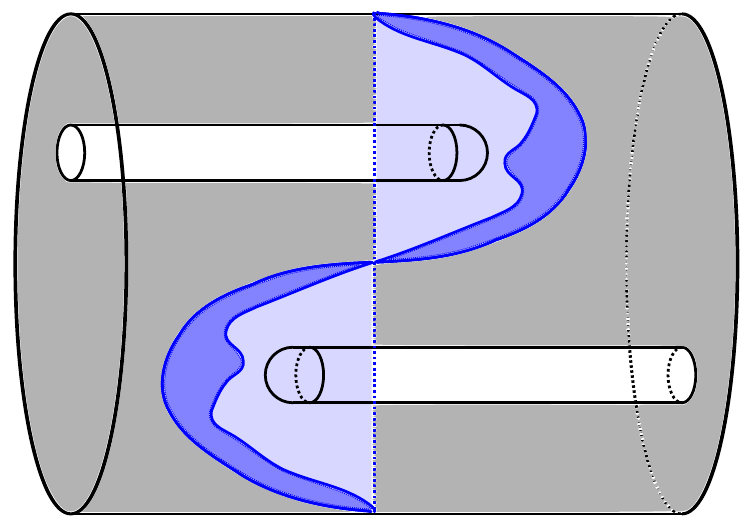}
		\caption{}
	\end{subfigure}
	\begin{subfigure}[t]{2.2in}
		\includegraphics[width=2.2in]{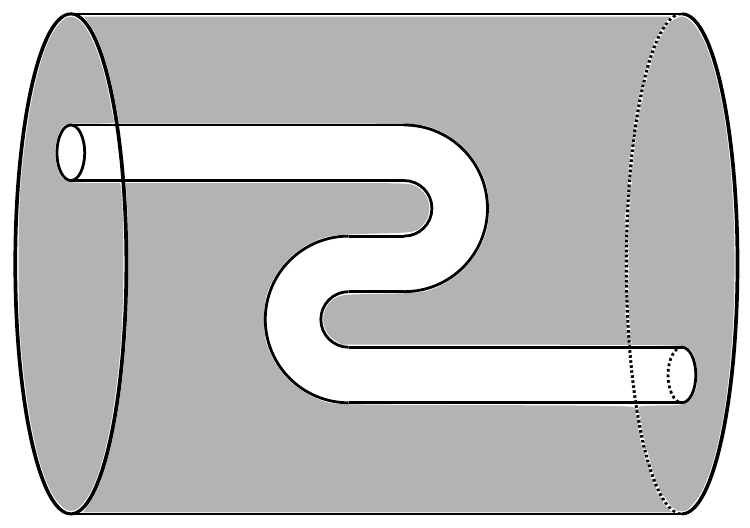}
		\caption{}
	\end{subfigure}
\end{center}
\caption{(a) A relative 2-cycle $\alpha$ from Theorem~7 of \cite{Coordinate-free} is depicted in blue. (b) Theorem~7 of \cite{Coordinate-free} proves that there is no evasion path in this sensor network. (c) Although there is no evasion path in this network, Theorem~7 of \cite{Coordinate-free} does not apply.}
\label{F: cartoonSheetTeleportSheet}
\end{figure}

We explain the picture behind this theorem. Suppose there is some
$$[\alpha] \in H_d(SC, F \times [0,1])$$
with $0 \neq [\partial \alpha]$. Let $\alpha$ be a relative $d$-cycle in $SC$ that represents the homology class $[\alpha]$. The condition $0 \neq [\partial \alpha]$ means  that the boundary of $\alpha$ wraps a nontrivial number of times around $F \times [0, 1]$. We think of $\alpha$ as a ``sheet'' in the region of spacetime covered by the sensors that separates time zero from time one; see Figure~\ref{F: cartoonSheetTeleportSheet}(a). If there is such a relative cycle $\alpha$ then no evasion path can exist. For example, Theorem~7 of \cite{Coordinate-free} proves there is no evasion path in the sensor network in Figure~\ref{F: cartoonSheetTeleportSheet}(b).

Theorem~7 of \cite{Coordinate-free} is equivalent to the following statement: if there is an evasion path in the sensor network, then every $[\alpha] \in H_d(SC, F \times [0,1])$ satisfies $0 = [\partial \alpha]$. This homological criterion is necessary but not sufficient for the existence of an evasion path. The insufficiency is demonstrated by the sensor network in Figure~\ref{F: cartoonSheetTeleportSheet}(c): every $[\alpha] \in H_d(SC, F \times [0,1])$ satisfies $0 = [\partial \alpha]$, but there is no evasion path since an intruder cannot move backwards in time. Can we sharpen this theorem to get necessary and sufficient conditions?

\section{Zigzag Persistence}\label{S: Zigzag Persistence}

We introduce zigzag persistence in this section before applying it to the evasion problem in the following section. Zigzag persistence \cite{ZigzagPersistence} is a generalization of persistent homology \cite{TopologicalPersistence, ComputingPersistent} in which the maps can go in either direction, and zigzag persistence is also the specific case of quiver theory \cite{Gabriel, DerksenWeyman} when the underlying quiver is a Dynkin diagram of type $A_n$.

A zigzag diagram is a directed graph with $n$ vertices and $n-1$ arrows
$$\bullet_1 \xleftrightarrow{} \bullet_2 \xleftrightarrow{} \bullet_3 \xleftrightarrow{} \ldots \xleftrightarrow{} \bullet_{n-1} \xleftrightarrow{} \bullet_n,$$
where each arrow points either to the left or to the right. Fix a field $k$. A zigzag module $V$ is a diagram
$$V_1 \xleftrightarrow{\ q_1\ } V_2 \xleftrightarrow{\ q_2\ } \ldots \xleftrightarrow{\ q_{n-2}\ } V_{n-1} \xleftrightarrow{\ q_{n-1}\ } V_n,$$
where each $V_i$ is a finite vector space over $k$ and each $q_i$ is a linear map pointing either to the left or to the right. A morphism $f$ between two zigzag modules $V$ and $W$ is a diagram
\begin{center}
\begin{tikzpicture}[description/.style={fill=white,inner sep=2pt}] 
\matrix (m) [matrix of math nodes, row sep=3em, 
column sep=2.5em, text height=1.5ex, text depth=0.25ex] 
{  
V_1 & V_2 & \ldots & V_{n-1} & V_n\\
W_1	& W_2 & \ldots & W_{n-1} & W_n\\
};
\path[<->,font=\scriptsize]
(m-1-1) edge node[auto] {} (m-1-2)
(m-1-2) edge node[auto] {} (m-1-3)
(m-1-3) edge node[auto] {} (m-1-4)
(m-1-4) edge node[auto] {} (m-1-5)
(m-2-1) edge node[auto] {} (m-2-2)
(m-2-2) edge node[auto] {} (m-2-3)
(m-2-3) edge node[auto] {} (m-2-4)
(m-2-4) edge node[auto] {} (m-2-5)
;
\path[->,font=\scriptsize]
(m-1-1) edge node[left] {$f_1$} (m-2-1)
(m-1-2) edge node[left] {$f_2$} (m-2-2)
(m-1-4) edge node[left] {$f_{n-1}$} (m-2-4)
(m-1-5) edge node[left] {$f_n$} (m-2-5)
;
\end{tikzpicture}
\end{center}
in which all of the squares commute. If each $f_i$ is an isomorphism of vector spaces, then $f$ is an isomorphism of zigzag modules. In the language of category theory, a zigzag module is a functor from the free category generated by a zigzag diagram to the category of finite vector spaces, and a morphism between two zigzag modules is a natural transformation \cite{MacLane}.

The direct sum of two zigzag modules $V$ and $V'$ is given by $(V \oplus V')_i = V_i \oplus V'_i$, with connecting linear maps of the form $q_i \oplus q'_i$. For birth and death indices $1 \leq b \leq d \leq n$, the interval module $\I(b,d)$ is defined by 
$$\I(b,d)_i = \begin{cases} k & \mbox{if } b \leq i \leq d \\ 0 & \mbox{otherwise.} \end{cases}
$$
The connecting linear maps of $\I(b,d)$ are identity maps $\1$ between adjacent copies of the field $k$, and zero maps otherwise.
So $\I(b,d)$ looks like
$$0 \xleftrightarrow{} \ldots \xleftrightarrow{} 0 \xleftrightarrow{} k \xleftrightarrow{\ \1\ } \ldots \xleftrightarrow{\ \1\ } k \xleftrightarrow{} 0 \xleftrightarrow{} \ldots \xleftrightarrow{} 0,$$
where the first $k$ is in slot $b$ and the last $k$ is in slot $d$. As in persistent homology, a zigzag module is described up to isomorphism by its barcode decomposition \cite{Gabriel, ZigzagPersistence}.

\begin{theorem}[Gabriel]\label{T: Gabriel}
A zigzag module $V$ can be decomposed as
$$ V \cong \oplus_{l=1}^N \I(b_l, d_l),$$
where the factors in the decomposition are unique up to reordering.
\end{theorem}

\noindent A barcode is a multiset of intervals of the form $[b, d]$, and the barcode for a zigzag module $ V \cong \oplus_{l=1}^N \I(b_l, d_l)$ is $\bigl\{[b_1, d_1], [b_2, d_2, ], \ldots, [b_N, d_N]\bigr\}$.

Given a fibrewise space $p \colon Y \to I$ and a choice of discretization
$$0 = s_0 < s_1 < \ldots < s_n = 1,$$
we build a zigzag module that models how the homology of fibrewise space $Y$ changes with time. Let $Y_i = p^{-1}(s_i)$ and let $Y_i^{i+1} = p^{-1}([s_i, s_{i+1}])$. We have the zigzag diagram
\begin{equation}\label{E: ZY}
Y_0 \hookrightarrow Y_0^1 \hookleftarrow Y_1 \hookrightarrow \ldots \hookleftarrow Y_{n-1} \hookrightarrow Y_{n-1}^n \hookleftarrow Y_n
\end{equation}
of topological spaces and inclusion maps \cite{Real-valued}. See Figure~\ref{F: cartoonYesNew_4} for an example.
\begin{figure}[h]
\begin{center}
    	\includegraphics[width=4in]{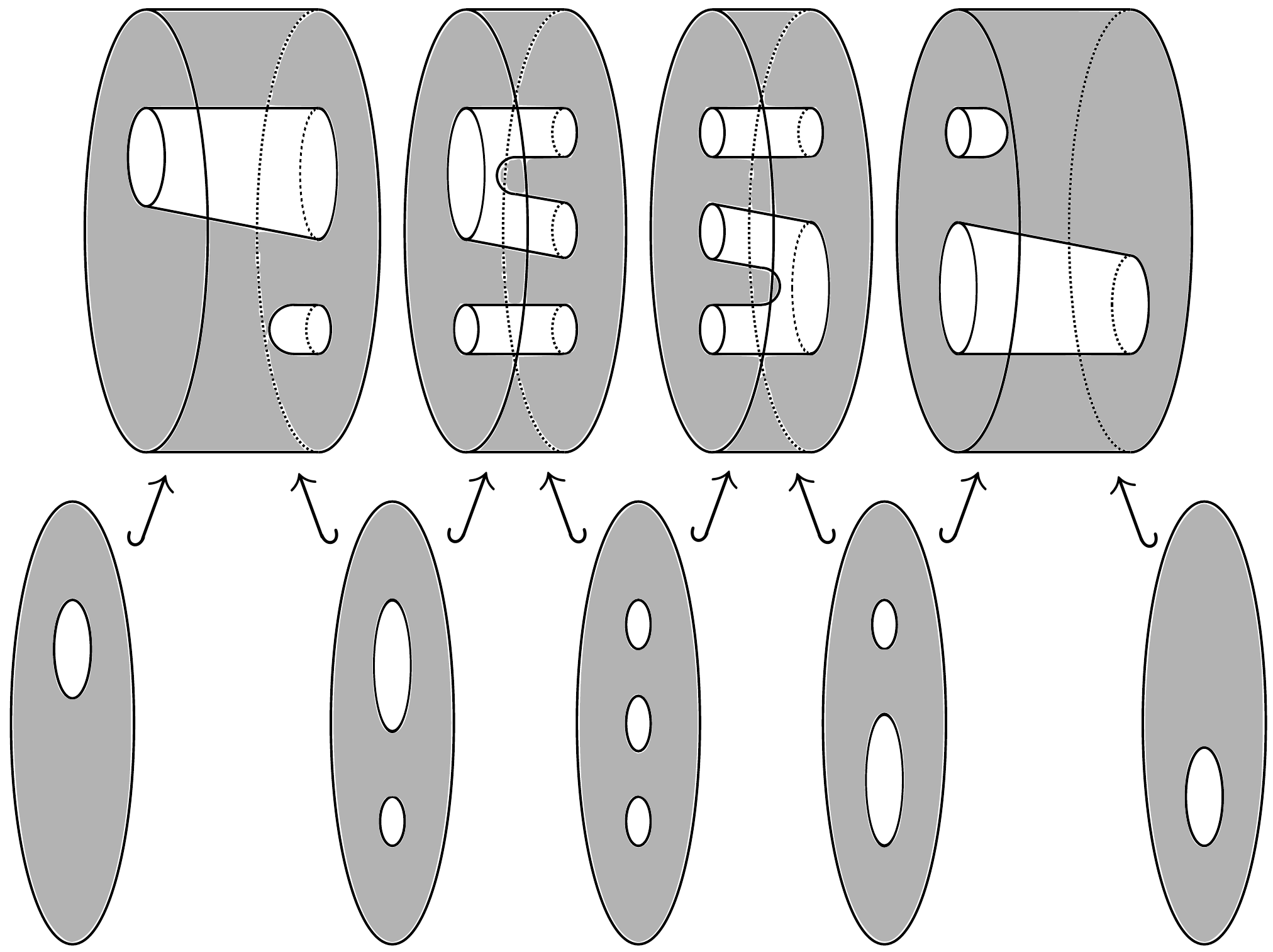}
\end{center}
\caption{A zigzag diagram built from the fibrewise space in Figure~\ref{F: cartoonYes}.}
\label{F: cartoonYesNew_4}
\end{figure}
We assume that each $Y_i$ and $Y_i^{i+1}$ have finite-dimensional homology and cohomology, each taken with coefficients in a field $k$. Applying the $j$-dimensional homology functor $H_j$ to (\ref{E: ZY}) gives the zigzag module
$$H_j(Y_0) \rightarrow H_j(Y_0^1) \leftarrow H_j(Y_1) \rightarrow \ldots \leftarrow H_j(Y_{n-1}) \rightarrow H_j(Y_{n-1}^n) \leftarrow H_j(Y_n).$$
We denote this zigzag module $ZH_j(Y)$, leaving implicit the choice of discretization and the choice of coefficient field. Applying the $j$-dimensional cohomology functor $H^j$ to (\ref{E: ZY}) gives the zigzag module
$$H^j(Y_0) \leftarrow H^j(Y_0^1) \rightarrow H^j(Y_1) \leftarrow \ldots \rightarrow H^j(Y_{n-1}) \leftarrow H^j(Y_{n-1}^n) \rightarrow H^j(Y_n),$$
which we denote $ZH^j(Y)$. Note the directions of the arrows have been reversed because cohomology is contravariant.

The following lemmas will be useful in the proof of Theorem~\ref{T: longBarcodeNecessary}. The first lemma states that zigzag persistent homology and cohomology are invariants of fibrewise homotopy type, and the second lemma states that the barcodes for zigzag persistent homology and cohomology are identical.

\begin{lemma}\label{L: invariance of ZH_j(-) and ZH^j(-)}
If $Y$ and $W$ are fibrewise homotopy equivalent then $ZH_j(Y) \cong ZH_j(W)$ and $ZH^j(Y) \cong ZH^j(W)$.
\end{lemma}

\begin{proof}
Let $f \colon Y \to W$ be a fibrewise homotopy equivalence. This induces the commutative diagram
\begin{center}
\begin{tikzpicture}[description/.style={fill=white,inner sep=2pt}] 
\matrix (m) [matrix of math nodes, row sep=3em, 
column sep=2.8em, text height=1.5ex, text depth=0.25ex] 
{  
Y_0 & Y_0^1 & Y_1 & \ldots & Y_{n-1} & Y_{n-1}^n & Y_n\\
W_0 & W_0^1 & W_1 & \ldots & W_{n-1} & W_{n-1}^n & W_n\\
};
\path[right hook->,font=\scriptsize]
(m-1-1) edge node[auto] {} (m-1-2)
(m-1-3) edge node[auto] {} (m-1-4)
(m-1-5) edge node[auto] {} (m-1-6)
(m-2-1) edge node[auto] {} (m-2-2)
(m-2-3) edge node[auto] {} (m-2-4)
(m-2-5) edge node[auto] {} (m-2-6)
;
\path[left hook->,font=\scriptsize]
(m-1-3) edge node[auto] {} (m-1-2)
(m-1-5) edge node[auto] {} (m-1-4)
(m-1-7) edge node[auto] {} (m-1-6)
(m-2-3) edge node[auto] {} (m-2-2)
(m-2-5) edge node[auto] {} (m-2-4)
(m-2-7) edge node[auto] {} (m-2-6)
;
\path[->,font=\scriptsize]
(m-1-1) edge node[left] {$f_0$} (m-2-1)
(m-1-2) edge node[left] {$f_0^1$} (m-2-2)
(m-1-3) edge node[left] {$f_1$} (m-2-3)
(m-1-5) edge node[left] {$f_{n-1}$} (m-2-5)
(m-1-6) edge node[left] {$f_{n-1}^n$} (m-2-6)
(m-1-7) edge node[left] {$f_n$} (m-2-7)
;
\end{tikzpicture}
\end{center}
where each map $f_i$ or $f_i^{i+1}$ is defined via restriction and is a homotopy equivalence. Since homology is a homotopy invariant, applying $H_j$ gives the commutative diagram
\begin{center}
\begin{tikzpicture}[description/.style={fill=white,inner sep=2pt}] 
\matrix (m) [matrix of math nodes, row sep=3em, 
column sep=1.2em, text height=1.5ex, text depth=0.25ex] 
{  
H_j(Y_0) & H_j(Y_0^1) & H_j(Y_1) & \ldots & H_j(Y_{n-1}) & H_j(Y_{n-1}^n) & H_j(Y_n)\\
H_j(W_0) & H_j(W_0^1) & H_j(W_1) & \ldots & H_j(W_{n-1}) & H_j(W_{n-1}^n) & H_j(W_n)\\
};
\path[->,font=\scriptsize]
(m-1-1) edge node[auto] {} (m-1-2)
(m-1-3) edge node[auto] {} (m-1-4)
(m-1-5) edge node[auto] {} (m-1-6)
(m-2-1) edge node[auto] {} (m-2-2)
(m-2-3) edge node[auto] {} (m-2-4)
(m-2-5) edge node[auto] {} (m-2-6)
;
\path[->,font=\scriptsize]
(m-1-3) edge node[auto] {} (m-1-2)
(m-1-5) edge node[auto] {} (m-1-4)
(m-1-7) edge node[auto] {} (m-1-6)
(m-2-3) edge node[auto] {} (m-2-2)
(m-2-5) edge node[auto] {} (m-2-4)
(m-2-7) edge node[auto] {} (m-2-6)
;
\path[->,font=\scriptsize]
(m-1-1) edge node[left] {} (m-2-1)
(m-1-2) edge node[left] {} (m-2-2)
(m-1-3) edge node[left] {} (m-2-3)
(m-1-5) edge node[left] {} (m-2-5)
(m-1-6) edge node[left] {} (m-2-6)
(m-1-7) edge node[left] {} (m-2-7)
;
\end{tikzpicture}
\end{center}
in which each vertical map is an isomorphism. Hence $ZH_j(Y) \cong ZH_j(W)$. The proof for cohomology is analogous.
\end{proof}

\begin{lemma}\label{L: naturality of (co)homology duality}
The barcodes for $ZH_j(Y)$ and $ZH^j(Y)$ are equal as multisets of intervals.
\end{lemma}

\begin{proof}
The version of this lemma with persistent homology instead of zigzag persistence is given in Proposition~2.3 of \cite{Dualities}, and our proof is analogous. Because their arrows point in different directions, the zigzag modules $ZH_j(Y)$ and $ZH^j(Y)$ live in different categories and cannot be isomorphic. However, consider the decomposition
$$ ZH_j(Y) \cong \oplus_{l=1}^N \I(b_l, d_l)$$
from Theorem~\ref{T: Gabriel}. Applying the contravariant functor $\Hom(\quad; k)$ produces the decomposition
$$\Hom(ZH_j(Y);k) \cong \oplus_{l=1}^N \I(b_l, d_l),$$
where the directions of the arrows in the zigzag modules have been reversed. Naturality of the Universal Coefficient Theorem \cite[Theorem~3.2]{Hatcher} with coefficients in a field gives $ZH^j(Y) \cong \Hom(ZH_j(Y);k)$, and hence the barcodes for $ZH_j(Y)$ and $ZH^j(Y)$ are equal as multisets of intervals.
\end{proof}

\section{Applying Zigzag Persistence to the Evasion Problem}\label{S: Applying Zigzag Persistence to the Evasion Problem}

We began studying the evasion problem with the goal of finding an if-and-only-if criterion for the existence of an evasion path using zigzag persistence, which in this setting describes how the homology of the region covered by sensors changes with time.

\begin{figure}[h]
\begin{center}
    	\includegraphics[width=6in]{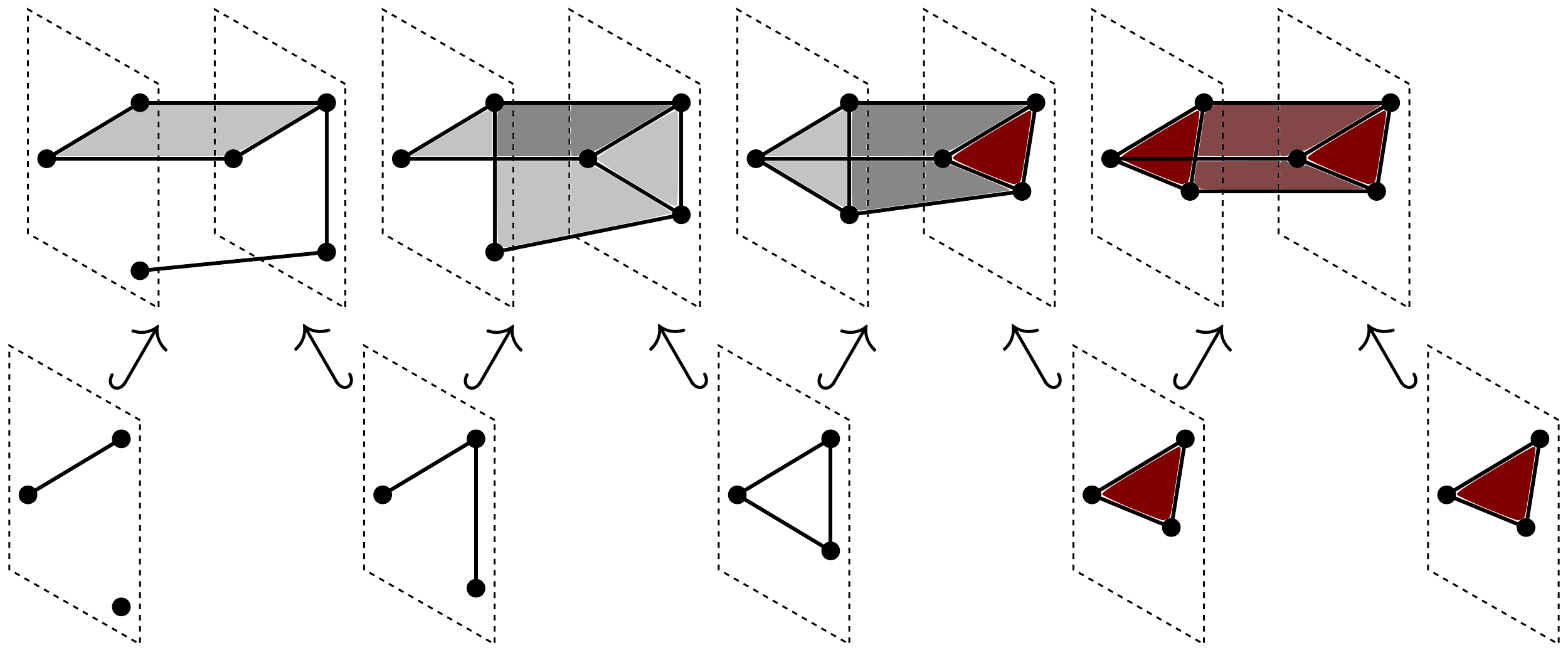}
\end{center}
\caption{The zigzag diagram for the stacked \u Cech complex from the three sensor nodes in Figure~\ref{F: stackedCech}.}
\label{F: stackedCechZigzag_arrows}
\end{figure}

Consider the times $0 < t_1 < \ldots < t_n < 1$
when the \u Cech complex changes and choose interleaving times
$$0 = s_0 < t_1 < s_1 < \ldots < t_n < s_n = 1.$$
We saw in Section~\ref{S: Zigzag Persistence} that the discretization $0 = s_0 < s_1 < \ldots < s_n = 1$ produces a zigzag diagram of spaces from any fibrewise space, and we will consider the fibrewise spaces $X$, $X^c$, and $SC$. Figure~\ref{F: stackedCechZigzag_arrows} depicts the zigzag diagram built from a stacked \u Cech complex $SC$.

\begin{lemma}\label{L: ZH_j(X) is ZH^j(SC)}
For $X$ the region of spacetime covered by sensors and $SC$ the stacked \u Cech complex, we have $ZH_j(X) \cong ZH_j(SC)$.
\end{lemma}

\begin{proof}
By the nerve lemma we have the following commutative diagram with each vertical arrow a homotopy equivalence.
\begin{center}
\begin{tikzpicture}[description/.style={fill=white,inner sep=2pt}] 
\matrix (m) [matrix of math nodes, row sep=3em, 
column sep=2.8em, text height=1.5ex, text depth=0.25ex] 
{  
SC_0 & SC_0^1 & SC_1 & \ldots & SC_{n-1} & SC_{n-1}^n & SC_n\\
X_0 & X_0^1 & X_1 & \ldots & X_{n-1} & X_{n-1}^n & X_n\\
};
\path[right hook->,font=\scriptsize]
(m-1-1) edge node[auto] {} (m-1-2)
(m-1-3) edge node[auto] {} (m-1-4)
(m-1-5) edge node[auto] {} (m-1-6)
(m-2-1) edge node[auto] {} (m-2-2)
(m-2-3) edge node[auto] {} (m-2-4)
(m-2-5) edge node[auto] {} (m-2-6)
;
\path[left hook->,font=\scriptsize]
(m-1-3) edge node[auto] {} (m-1-2)
(m-1-5) edge node[auto] {} (m-1-4)
(m-1-7) edge node[auto] {} (m-1-6)
(m-2-3) edge node[auto] {} (m-2-2)
(m-2-5) edge node[auto] {} (m-2-4)
(m-2-7) edge node[auto] {} (m-2-6)
;
\path[->,font=\scriptsize]
(m-1-1) edge node[left] {} (m-2-1)
(m-1-2) edge node[left] {} (m-2-2)
(m-1-3) edge node[left] {} (m-2-3)
(m-1-5) edge node[left] {} (m-2-5)
(m-1-6) edge node[left] {} (m-2-6)
(m-1-7) edge node[left] {} (m-2-7)
;
\end{tikzpicture}
\end{center}
The remainder of the proof is identical to the proof of Lemma~\ref{L: invariance of ZH_j(-) and ZH^j(-)}.
\end{proof}

Our initial hypothesis was that an evasion path would exist in a sensor network if and only if there were a full-length interval $[1, 2n+1]$ in the barcode for $ZH_{d-1}(SC)$. For example, network (a) in Figure~\ref{F: cartoonYesTeleport_barcode} has both an evasion path and a full-length interval, and network (b) has neither. Only one direction of this hypothesis is true.

\begin{figure}[h]
\begin{center}
	\begin{subfigure}[t]{2.2in}
		\includegraphics[width=2.2in]{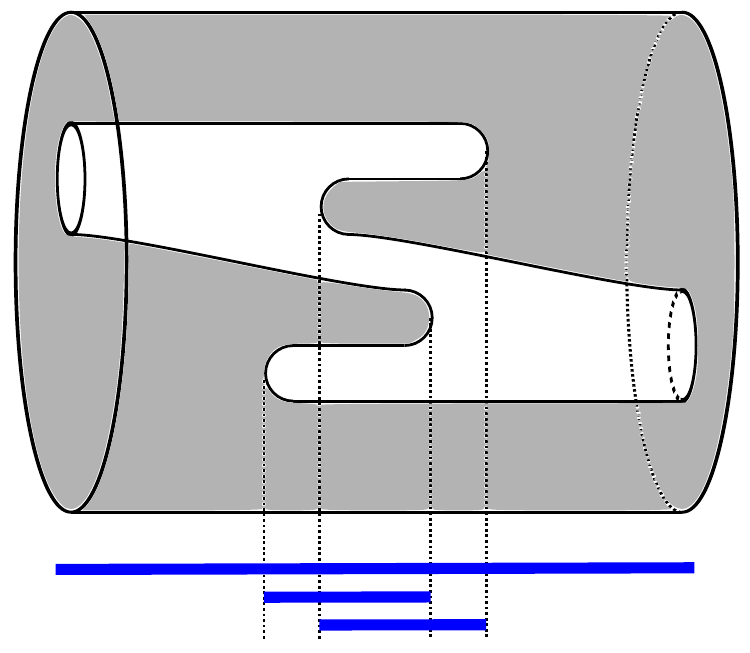}
		\caption{}
	\end{subfigure}
	\hspace{10mm}
	\begin{subfigure}[t]{2.2in}
		\includegraphics[width=2.2in]{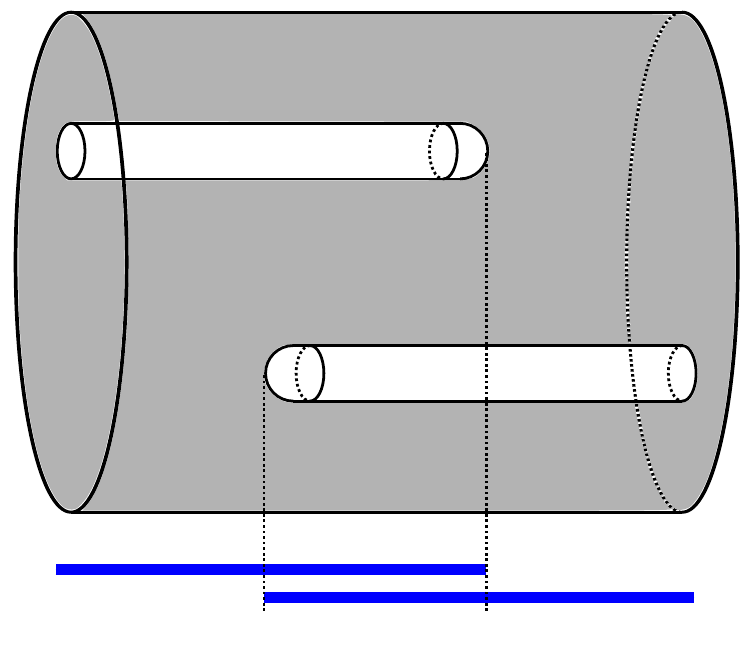}
		\caption{}
	\end{subfigure}
\end{center}
\caption{Two planar sensor networks and their barcode decompositions for $ZH_1(X)$.}
\label{F: cartoonYesTeleport_barcode}
\end{figure}

\begin{theorem}\label{T: longBarcodeNecessary} 
If there is an evasion path in a sensor network, then there is a full-length interval $[1, 2n+1]$ in the zigzag barcode for $ZH_{d-1}(SC)$.
\end{theorem}

\begin{proof}
An evasion path is a section $s \colon I \to X^c$, that is, a commutative diagram
\begin{center}
\begin{tikzpicture}[description/.style={fill=white,inner sep=2pt}] 
\matrix (m) [matrix of math nodes, row sep=3em, 
column sep=2.5em, text height=1.5ex, text depth=0.25ex] 
{  
I	&X^c		&I\\};
\path[->, font=\scriptsize]
(m-1-1) edge node[auto] {$ s $} (m-1-2)
(m-1-2) edge node[auto] {$ p $} (m-1-3)
;
\path[->, font=\scriptsize, bend right]
(m-1-1) edge node[below] {$ \1$} (m-1-3)
;
\end{tikzpicture}
\end{center}
with $\1$ the identity map. Applying zigzag homology $ZH_0$ gives the following commutative diagram.
\begin{center}
\begin{tikzpicture}[description/.style={fill=white,inner sep=2pt}] 
\matrix (m) [matrix of math nodes, row sep=3em, 
column sep=2.5em, text height=1.5ex, text depth=0.25ex] 
{  
ZH_0(I)	&ZH_0(X^c)	&ZH_0(I)\\};
\path[->, font=\scriptsize]
(m-1-1) edge node[above] {$ZH_0(s)$} (m-1-2)
(m-1-2) edge node[above] {$ZH_0(p)$} (m-1-3)
;
\path[->, font=\scriptsize, bend right]
(m-1-1) edge node[below] {$\1$} (m-1-3)
;
\end{tikzpicture}
\end{center}
Since the identity map on $ZH_0(I) \cong \I(1, 2n+1)$ factors through $ZH_0(X^c)$, the splitting lemma \cite[Section~2.2]{Hatcher}
implies that the barcode decomposition for $ZH_0(Y)$ contains a summand isomorphic to $\I(1, 2n+1)$. Hence there is a full-length interval $[1, 2n+1]$ in the barcode for $ZH_0(X^c)$. Next we need a version of Alexander Duality \cite[Theorem~3.44]{Hatcher}. We apply Theorem~3.11 of \cite{Kalisnik}, which uses the Diamond Principle of \cite{ZigzagPersistence} and our Lemma~\ref{L: naturality of (co)homology duality}, to obtain a full-length interval in $ZH_{d-1}(X)$. Since $ZH_{d-1}(X) \cong ZH_{d-1}(SC)$ by Lemma~\ref{L: ZH_j(X) is ZH^j(SC)}, the proof is complete.
\end{proof}

\begin{remark}
This theorem is as discerning as the reformulated version of Theorem~7 of \cite{Coordinate-free}. That is, one theorem can be used to prove that no evasion path exists in a sensor network if and only if the other theorem can be used. However, suppose that the sensors move for a long period of time. In this case the amalgamated complex used in Corollary~3 of \cite{Coordinate-free} to compute their homological criterion may become quite large. By contrast, the algorithm for computing zigzag persistence runs in a streaming fashion that does not require storing the sensor network across all times simultaneously \cite{Real-valued}. Hence computing our Theorem~\ref{T: longBarcodeNecessary} may be more feasible for sensors moving over a long period of time.
\end{remark}

Interestingly, the converse to Theorem~\ref{T: longBarcodeNecessary} is false. This is demonstrated by the sensor network in Figure~\ref{F: cartoonNo_explanation}(a). It is tempting to guess that the barcode for this network consists of the intervals drawn on top in black, but they are crossed out because they are incorrect. The correct barcode beneath contains a full-length interval $[1, 2n+1]$ even though there is no evasion path. We explain this counterintuitive barcode in Figure~\ref{F: cartoonNo_explanation}(b), and we give a second explanation in Section~\ref{S: Dependence on the Embedding}.

\begin{figure}[h]
\begin{center}
	\begin{subfigure}[t]{2.2in}
		\includegraphics[width=2.2in]{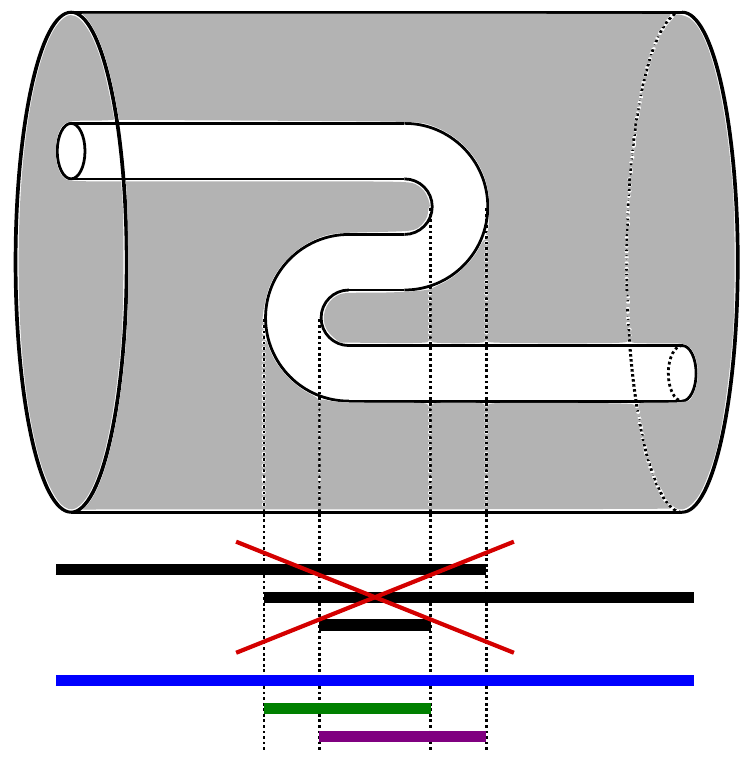}
		\caption{}
	\end{subfigure}
	\begin{subfigure}[t]{6in}
		\includegraphics[width=6in]{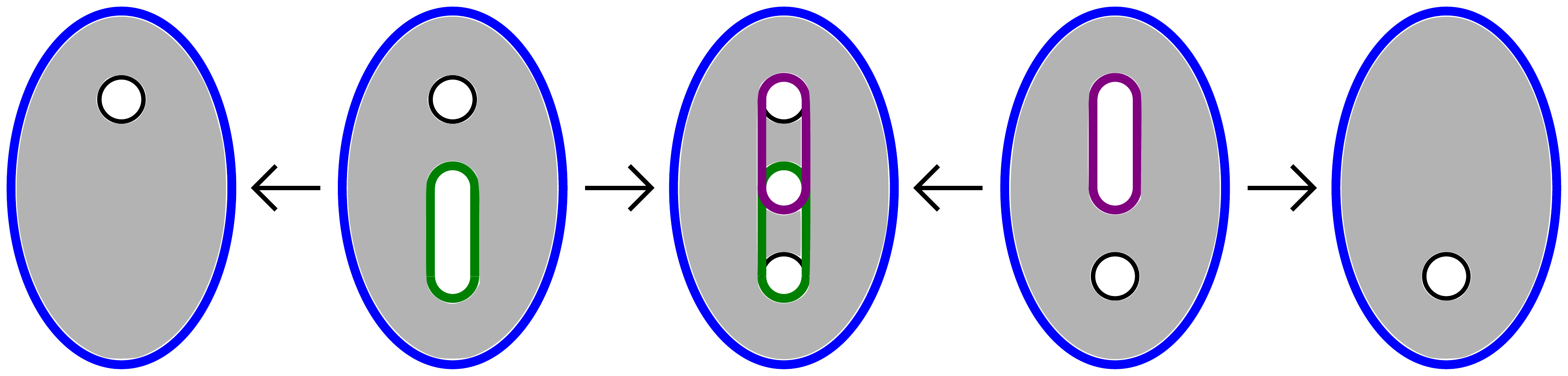}
		\caption{}
	\end{subfigure}
\end{center}
\caption{(a) It is tempting to guess that the barcode for $ZH_1(X)$ consists of the crossed-out intervals on top in black, but instead the correct barcode is drawn beneath in blue, green, and purple. Note there is a full-length interval $[1, 2n+1]$ even though there is no evasion path in this network. (b) A coarsened version of the zigzag diagram for $X$. The cycles drawn in blue, green, and purple are generators for the three intervals in $ZH_1(X)$.}
\label{F: cartoonNo_explanation}
\end{figure}

Caution~2.9 from \cite{ZigzagPersistence} explains that although every submodule isomorphic to an interval in a persistent homology module corresponds to a direct summand, the same is not true for zigzag modules. The sensor networks in Figures~\ref{F: cartoonYesTeleport_barcode}(b) and \ref{F: cartoonNo_explanation}(a) are good examples of this caution. The zigzag modules for both sensor networks have a submodule isomorphic to the full-length interval module $\I(1, 2n+1)$, but Figure~\ref{F: cartoonNo_explanation}(a) contains $\I(1, 2n+1)$ as a summand whereas Figure~\ref{F: cartoonYesTeleport_barcode}(b) does not.

\FloatBarrier
\section{Dependence on the Embedding}\label{S: Dependence on the Embedding}

It turns out that the answer to the evasion problem is no: in general, neither the time-varying \u Cech complex $C(t)$ of a sensor network nor the fibrewise homotopy type of covered region $X$ determine if an evasion path exists. The ambient isotopy class of the fibrewise embedding of $X$ into spacetime $\xD \times I$ also matters.

\begin{figure}[h]
\begin{center}
	\begin{subfigure}[t]{6in}
		\includegraphics[width=6in]{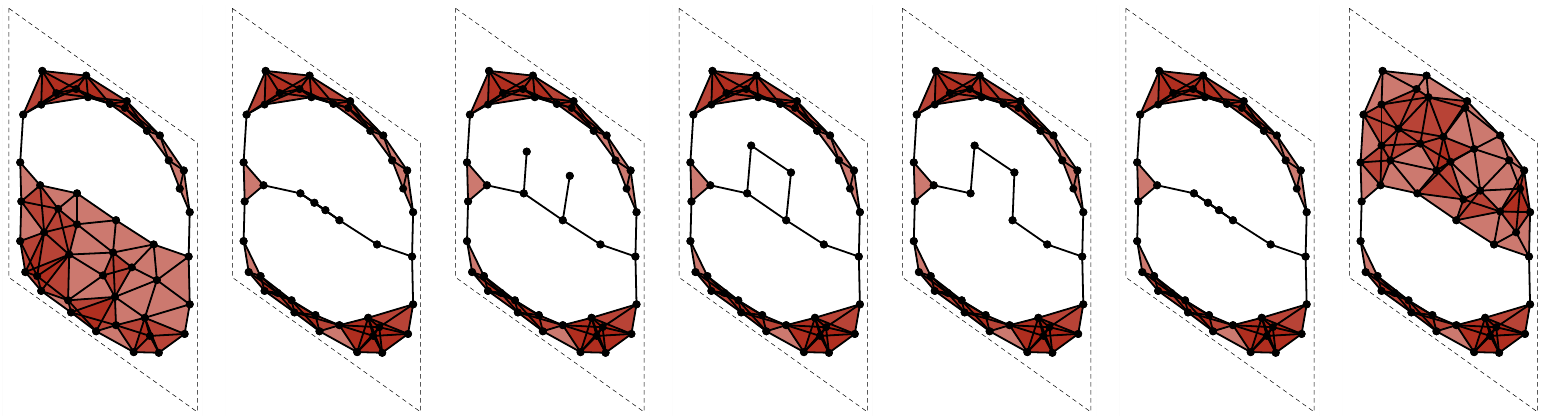}
	\end{subfigure}	
	\begin{subfigure}[t]{2.0in}		
		\includegraphics[width=2.0in]{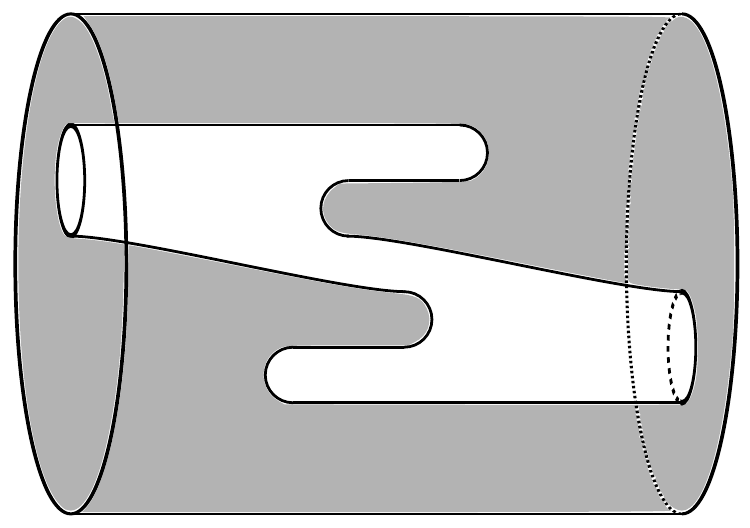}
		\caption{}	
	\end{subfigure}
	\begin{subfigure}[t]{6in}
		\includegraphics[width=6in]{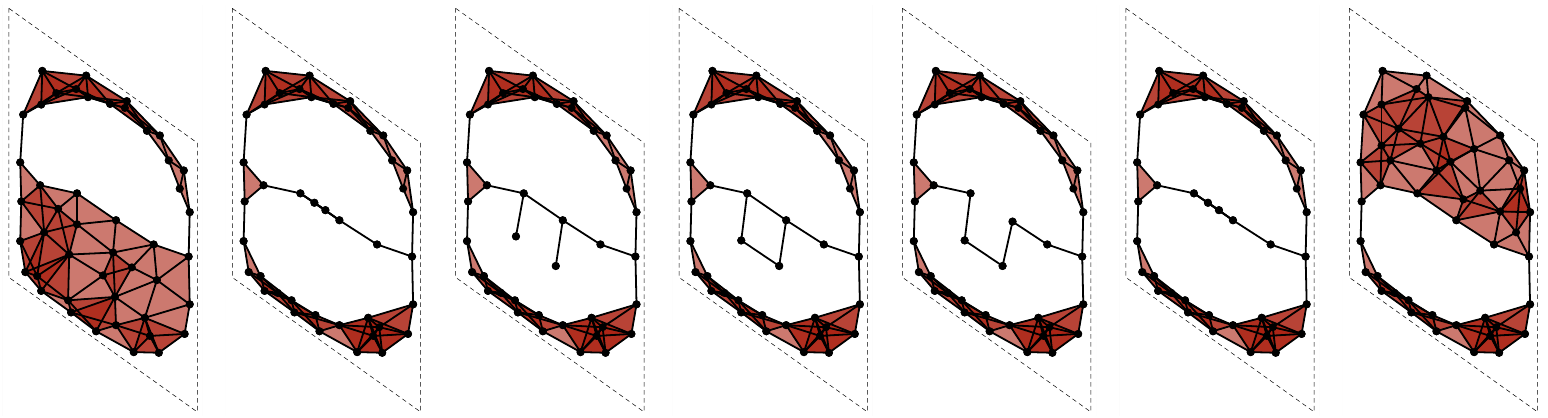}
	\end{subfigure}
	\begin{subfigure}[t]{2.0in}
		\includegraphics[width=2.0in]{cartoonNo.pdf}
		\caption{}
	\end{subfigure}
\end{center}
\caption{Each subfigure is a sensor network represented both as seven sequential \u Cech complexes and as a covered region $X$ in spacetime $\xD \times I$. At each time $t \in I$ the \u Cech complexes $C(t)$ in (a) and (b) are identical. Moreover, the two covered regions are fibrewise homotopy equivalent. Nevertheless, network (a) contains an evasion path, but network (b) does not because the intruder cannot travel backwards in time. See \href{http://www.ima.umn.edu/~henrya/research/DependenceOnTheEmbedding.avi}{Extension~1} for a video of these sensor networks.}
\label{F: cartoonYesNo}
\end{figure}

We demonstrate this impossibility result using the planar sensor networks (a) and (b) in Figure~\ref{F: cartoonYesNo}. Let us describe network (a). Initially, the bottom half of domain $\xD$ is covered by sensors. These sensors retreat to the boundary, leaving a horizontal line of sensors. Two sensors on this line jut out towards the top of $\xD$, forming three sides of a square. These two sensors move closer together, completing the square. The bottom two sensors in this square move apart, breaking the bottom edge of the square. The curvy line of sensors straightens out. Finally, sensors flood from the boundary to cover the top half of $\xD$. Network (b) is identical to (a) except that the square opens towards the bottom of $\xD$. For a video of these sensor networks, see \href{http://www.ima.umn.edu/~henrya/research/DependenceOnTheEmbedding.avi}{Extension~1}.

The time-varying \u Cech complexes $C(t)$ for the two networks are identical, but network (a) has an evasion path while network (b) does not. Furthermore, the covered regions for these two networks are fibrewise homotopy equivalent, and the stacked \u Cech complexes are fibrewise homeomorphic. However, the uncovered regions for the networks are not fibrewise homotopy equivalent, and in particular (a) has a section while (b) does not. Thus the existence of an evasion path depends not only on the fibrewise homotopy type of the sensor network but also on how the sensor network is fibrewise embedded in spacetime $\xD \times I$. Morover, in domain $\xD \subset \R^d$ for any $d \geq 2$ there exists a pair of analogous sensor networks\footnote{Analogous examples in $\xD \subset \R^1$ require the sensors to turn off and then back on.}.

\begin{remark}
In the static setting in which the sensors do not move, one can use the \u Cech complex to determine if the sensors cover the entire domain $\xD$. However, in the setting of mobile sensors, the time-varying \u Cech complex does not in general determine if there exists an evasion path or not.
\end{remark}

In Section~\ref{S: Applying Zigzag Persistence to the Evasion Problem} we promised a second explanation for the counterintuitive zigzag barcode in Figure~\ref{F: cartoonNo_explanation}(a), which we give now. The sensor network in Figure~\ref{F: cartoonNo_explanation}(a) is the same as the network in Figure~\ref{F: cartoonYesNo}(b), whose covered region is fibrewise homotopy equivalent to the covered region for the network in Figure~\ref{F: cartoonYesNo}(a). Hence by Lemma~\ref{L: invariance of ZH_j(-) and ZH^j(-)} their zigzag barcodes must be equal, and the zigzag barcodes for the network in Figure~\ref{F: cartoonYesNo}(a) are shown in Figure~\ref{F: cartoonYesTeleport_barcode}(a).

Embeddings are a central theme in topology. For topological spaces $X$ and $Y$, an embedding $f \colon X \hookrightarrow Y$ maps $X$ injectively and homeomorphically onto its image. A typical goal is to classify the space of embeddings up to some notion of equivalence, such as isotopy or ambient isotopy, and the difficulty of this task depends heavily on the spaces $X$ and $Y$. Knot theory considers the case when $X$ is the circle and $Y = \R^3$, and higher dimensional analogues are even more complicated. However, for some choices of manifolds $X$ and $Y$, homotopy based classifications for the space of embeddings do exist \cite{Whitney, Haefliger, Adachi}. For the evasion problem (and also its natural extension in which one would like to describe not only whether an evasion path exists but also the entire space of evasion paths), it would be useful to have extensions of embedding theory both to the setting of non-manifold spaces and to the setting of fibrewise spaces. One possibility is to try to adapt the tools of embedding calculus \cite{Embeddings} to a fibrewise setting.

\FloatBarrier
\section{Sensors Measuring Cyclic Orderings}\label{S: Sensors Measuring Cyclic Orderings}

Since neither the time-varying \u Cech complex nor the fibrewise homotopy type of covered region $X$ are sufficient to determine if an evasion path exists, what minimal sensing capabilities might we add? In this section we assume the sensors live in a planar domain $\xD \subset \R^2$ and that each sensor measures the cyclic ordering of its neighbors, as in \cite{SurroundingNodes}. It is not uncommon for sensors to measure this weak angular data, for example by performing circular radar sweeps. In Theorem~\ref{T: Stronger sensors} we give necessary and sufficient conditions for the existence of an evasion path based on this rotation information.

Theorem~\ref{T: Stronger sensors} relies on the alpha complex of the sensors. Let $V_{v(t)}$ be the Voronoi cell
$$V_{v(t)} = \{ y \in \xD\ |\ \|v(t) - y\| \leq \|\tilde{v}(t) - y\| \mbox{ for all } \tilde{v} \in S\}$$
of all points in $\xD$ closest to sensor $v$ at time $t$. The alpha complex $A(t)$ is the nerve of the convex sets $\{ B_{v(t)} \cap V_{v(t)} \}_{v \in S}$ \cite{AlphaShapes, ComputationalTopology}. It is a subcomplex of both the \u Cech complex and of the Delaunay triangulation, and is homotopy equivalent to the \u Cech complex and to the union of the sensor balls. For points in general position the 1-skeleton of the alpha complex is embedded in the plane, though the 1-skeleton of the \u Cech complex need not be. Recovering the alpha complex instead of the \u Cech complex requires significantly stronger sensors. However, if each sensor measures the local distances to its overlapping neighbors, which may be approximated by time-of-flight, then this data determines the alpha complex \cite{FayedMouftah}.

\begin{figure}[h]
\begin{center}
	\begin{subfigure}[t]{3.2in}
		\includegraphics[width=3.2in]{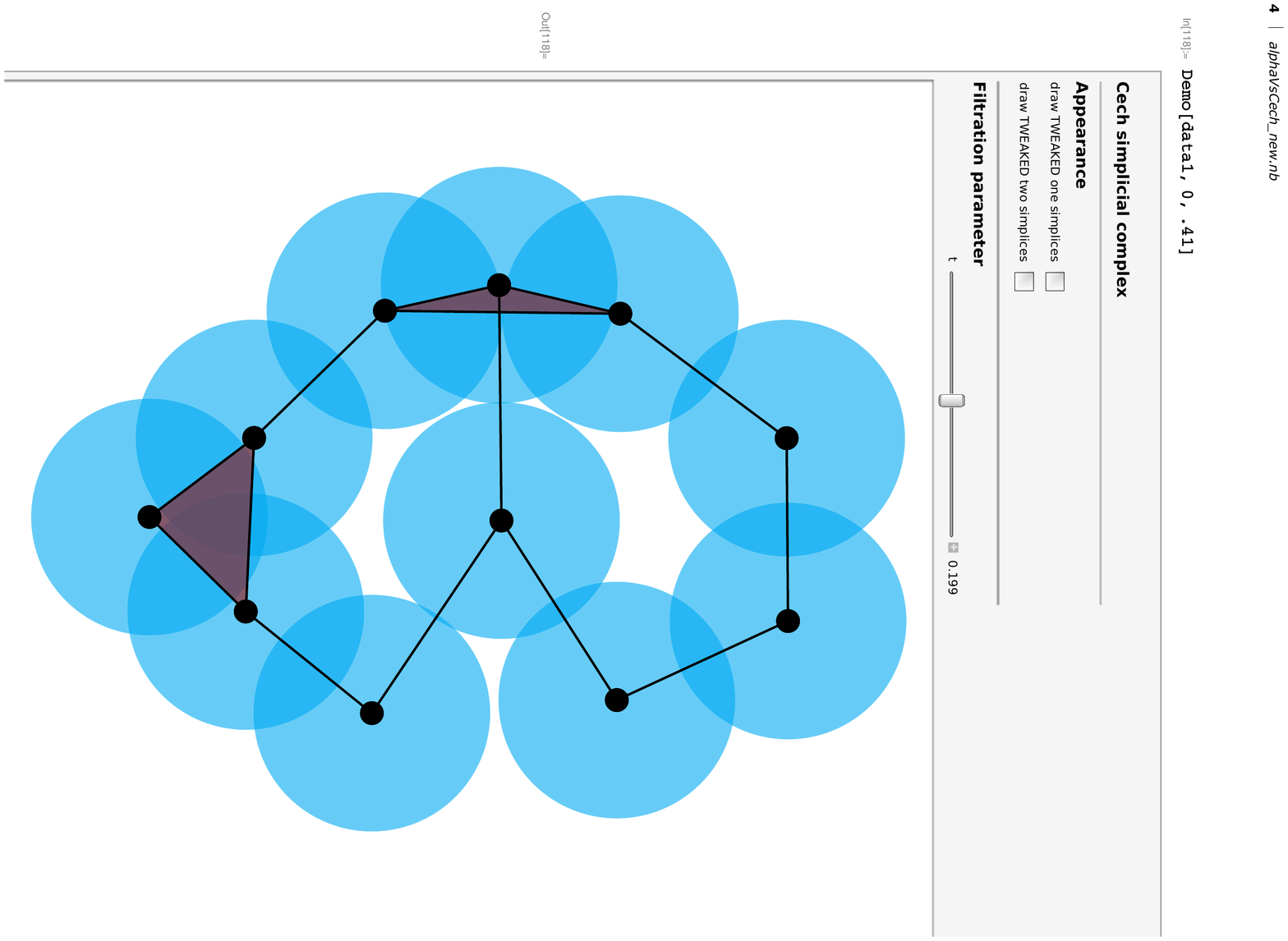}
		\caption{\u Cech complex}
	\end{subfigure}
	\begin{subfigure}[t]{3.2in}
		\includegraphics[width=3.2in]{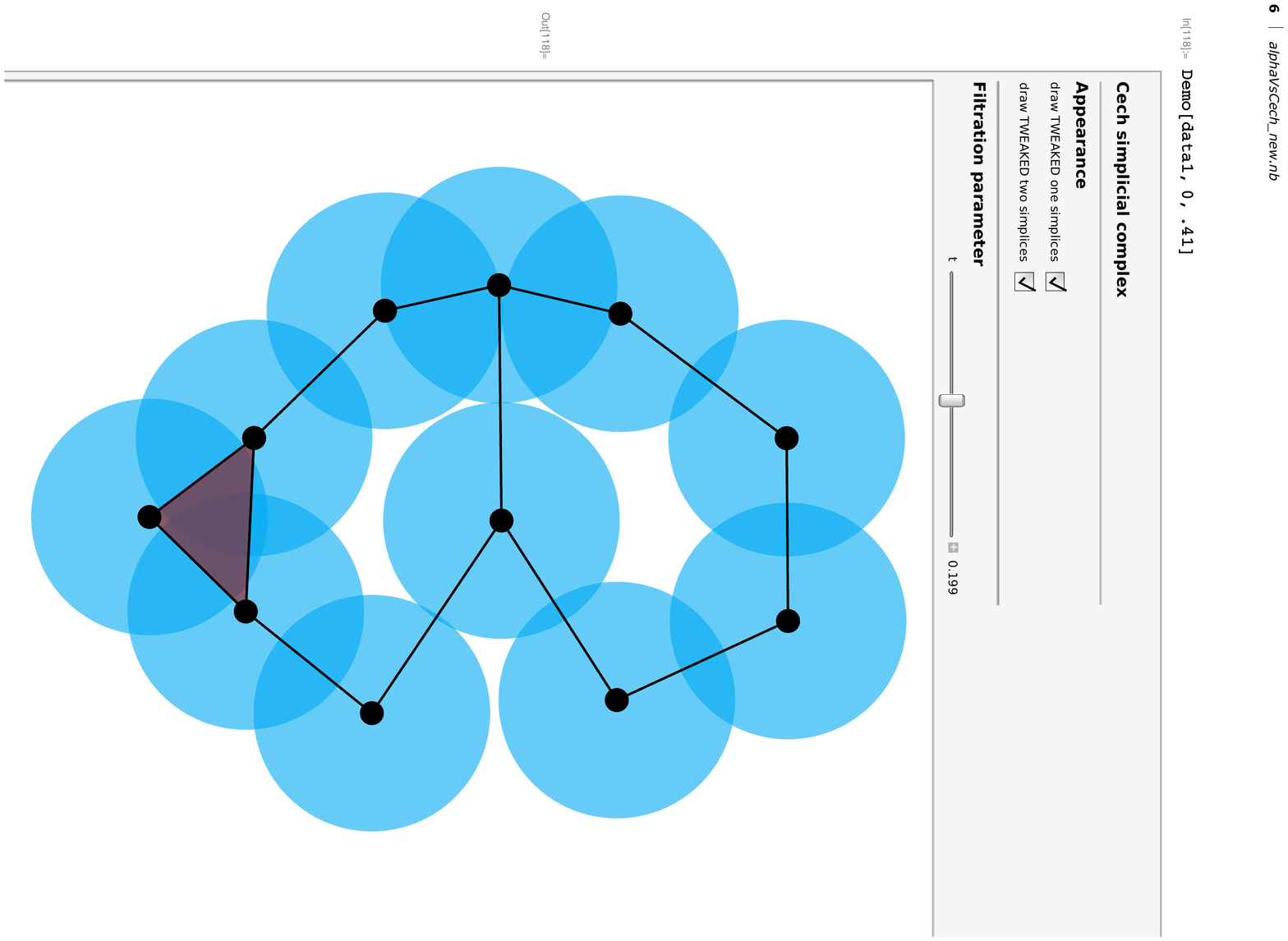}
		\caption{Alpha complex}
	\end{subfigure}
\end{center}
\caption{The alpha complex is a homotopy equivalent subcomplex of the \u Cech complex. For points in general position the 1-skeleton of the alpha complex is embedded in the plane, but the 1-skeleton of the \u Cech complex need not be.}
\label{F: alphaVsCech}
\end{figure}

We assume there are only a finite number of times $0 < t_1 < \ldots < t_n < 1$ when the alpha complex changes. Hence for $t$ and $t'$ in $(t_i, t_{i+1})$, $[0, t_1)$, or $(t_n, 1]$ we have identical alpha complexes $A(t) = A(t')$. Moreover, we assume that at each $t_i$ one of the following changes to the alpha complex occurs.
\begin{enumerate}
\item A single edge is added or removed.
\item A single 2-simplex is added or removed.
\item A free pair consisting of a 2-simplex and a face edge with no other cofaces is added or removed.
\item A Delaunay edge flip occurs.
\end{enumerate}
We assume that each sensor measures the clockwise cyclic ordering of its neighbors in the alpha complex. This cyclic ordering data is necessarily fixed in each interval $(t_i, t_{i+1})$, $[0, t_1)$, or $(t_n, 1]$.

\begin{theorem}\label{T: Stronger sensors}
Suppose we have a planar sensor network with covered region $X(t)$ connected at each time $t \in I$. Then from the time-varying alpha complex and the time-varying cyclic orderings of the neighbors about each sensor, we can determine whether or not an evasion path exists.
\end{theorem}

\begin{proof}
Let $A^1(t)$ be the 1-skeleton of the alpha complex at time $t$. For each vertex $v$ we have a cyclic permutation $\pi_v$ acting on the incident edges, where $\pi_v(e)$ is the successor of edge $e$ in the clockwise ordering around $v$. This gives $A^1(t)$ the structure of a rotation system \cite{MoharThomassen}, also called a fat graph or ribbon graph \cite{Igusa}. A rotation system partitions the directed edges of $A^1(t)$ into sets of boundary cycles. Each boundary cycle is a loop of directed edges $(e_1 e_2 \ldots e_k)$ constructed so that if $v_i$ is the target vertex of directed edge $e_i$, then $\pi_{v_i}(e_i) = e_{i+1}$, where $e_{k+1} = e_1$. See Figure~\ref{F: fatGraphUncanonical} for an example, and note that this cyclic ordering data distinguishes the two sensor networks in Figure~\ref{F: cartoonYesNo}.

\begin{figure}[h]
\begin{center}
    	\includegraphics[width=4in]{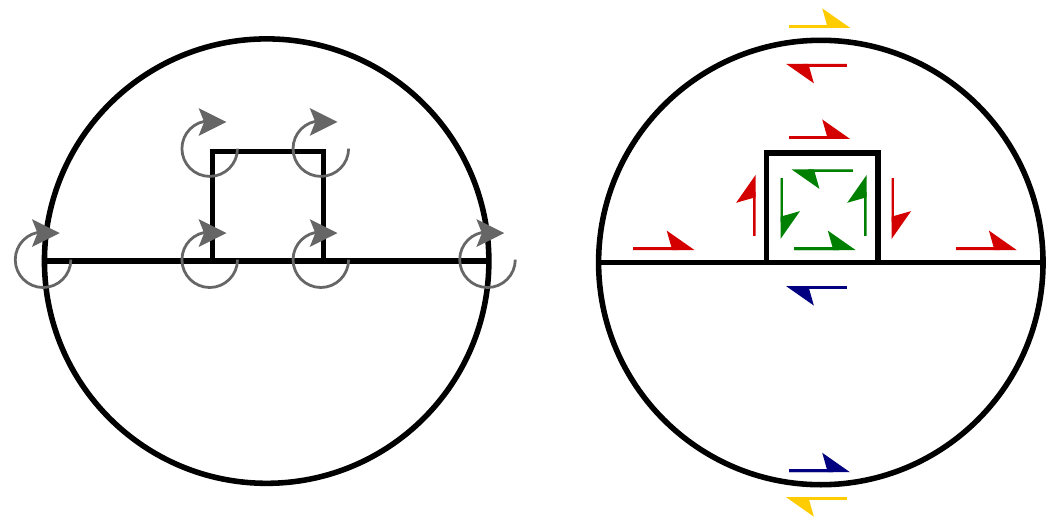}
\end{center}
\caption{An example rotation system. The cyclic orderings are drawn on the left in gray, and the four boundary cycles are drawn on the right in red, green, blue, and yellow.}
\label{F: fatGraphUncanonical}
\end{figure}

The boundary cycles of $A^1(t)$ are in bijective correspondence with the connected components of $\R^2 \setminus A^1(t)$. Removing the boundary cycles of length three that are filled by 2-simplices (and also the boundary cycle corresponding to the outside of $\partial \xD$) produces a bijection with the connected components of the uncovered region $X(t)^c$. Hence by tracking the boundary cycles of $A^1(t)$ we can measure how the connected components of the uncovered region merge, split, appear, and disappear. In other words, we can reconstruct the Reeb graph of $X^c \to I$ \cite{Reeb}.

We will maintain labels on the boundary cycles of $A^1(t)$ so that a boundary cycle is labeled {\em true} if the corresponding connected component of $\R^2 \setminus A^1(t)$ may contain an intruder and {\em false} if not. At time $t=0$ we label the boundary cycles of length three filled by 2-simplices in $A(0)$ (and also the boundary cycle corresponding to the outside of $\partial \xD$) as {\em false}. All other boundary cycles are labeled {\em true}. When we pass a time $t_i$ when the alpha complex changes, we update the labels as follows.
\begin{enumerate}
\item If a single edge is added, then a single boundary cycle splits in two since $X(t)$ is connected. Each new boundary cycle maintains the original label. If a single edge is removed, then two boundary cycles merge since $X(t)$ remains connected, and the new cycle is labeled {\em true} if either of the original two cycles were labeled {\em true}.
\item If a single 2-simplex is added, then the label on the corresponding boundary cycle of length three is set to {\em false}. If a single 2-simplex is removed, then the label on the corresponding boundary cycle of length three remains {\em false}.
\item If a free pair consisting of a 2-simplex and a face edge is added, then a boundary cycle splits into two with one label unchanged. The other label corresponding to the added 2-simplex is set to {\em false}. If a free pair is removed, then the boundary cycle of length three corresponding to the 2-simplex is removed and the label on the other modified boundary cycle remains unchanged.
\item If a Delaunay edge flip occurs, then two boundary cycles labeled {\em false} are replaced by two different boundary cycles also labeled {\em false}.
\end{enumerate}
An evasion path exists in the sensor network if and only if there is a boundary cycle in $A^1(1)$ labeled {\em true}. Such a boundary cycle corresponds to a connected component of the uncovered region $X(1)^c$ at time one which could contain an intruder.
\end{proof}

\begin{figure}[h]
\begin{center}
	\begin{subfigure}[t]{6in}
		\includegraphics[width=6in]{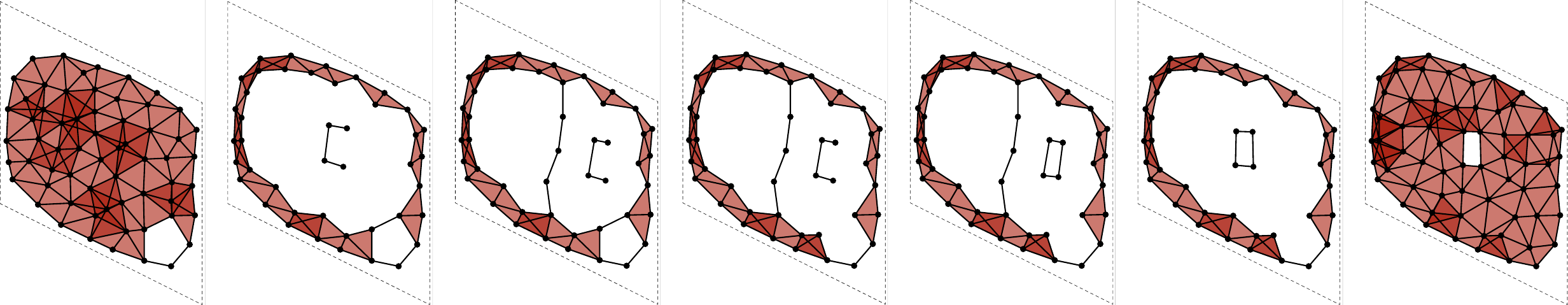}
		\caption{}
	\end{subfigure}	
	\begin{subfigure}[t]{6in}
		\includegraphics[width=6in]{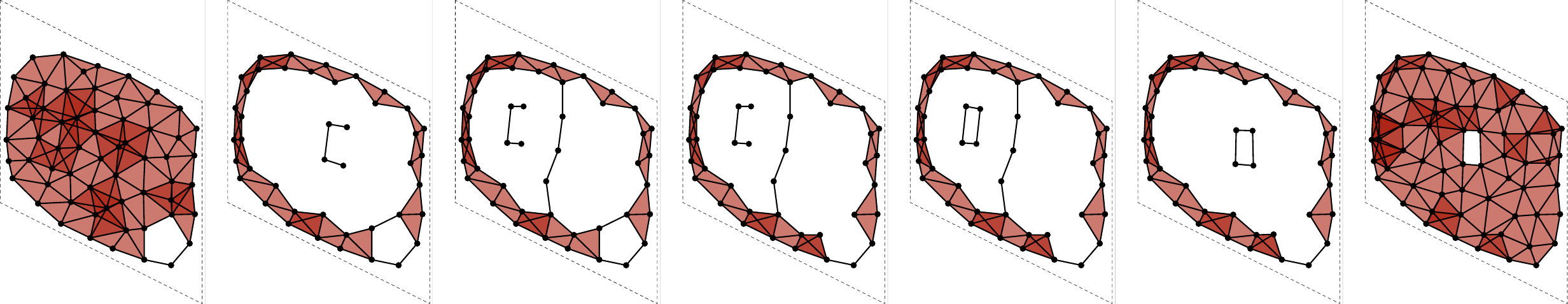}
		\caption{}
	\end{subfigure}
\end{center}
\caption{Each subfigure is a sensor network represented by seven sequential \u Cech complexes. At each time the \u Cech complexes, alpha complexes, and cyclic ordering information are identical. Nevertheless, network (a) contains an evasion path but network (b) does not. Hence it is necessary in Theorem~\ref{T: Stronger sensors} to assume that each $X(t)$ is connected.}
\label{F: needConnectedYesNo}
\end{figure}

Figure~\ref{F: needConnectedYesNo} shows that the connectedness assumption in Theorem~\ref{T: Stronger sensors} is necessary. It is an open question if the cyclic ordering information along with the time-varying \u Cech complex (instead of the time-varying alpha complex) suffice.

\begin{open question}
Suppose we have a planar sensor network with $X(t)$ connected at each time $t$. Using only the time-varying \u Cech complex and the time-varying cyclic orderings of the neighbors about each sensor, is it possible to determine if an evasion path exists?
\end{open question}

An answer to this open question would fill the gap between Theorem~7 of \cite{Coordinate-free} or equivalently our Theorem~\ref{T: longBarcodeNecessary}, which use only minimal sensor capabilities but are not sharp, and our Theorem~\ref{T: Stronger sensors}, which is sharp but requires more advanced sensors measuring alpha complexes. One difficulty in working with the \u Cech complex is that its 1-skeleton need not be embedded in the plane; see Figure~\ref{F: alphaVsCech}(a).

\section{Conclusions}\label{C: Conclusions}

This paper addresses an evasion problem for mobile sensor networks in which the sensors don't know their locations and instead measure only local connectivity data. In \cite{Coordinate-free}, de~Silva and Ghrist  provide a homological criterion depending on this limited input which rules out the existence of an evasion path in many sensor networks. We use zigzag persistence to produce a criterion of equivalent discriminatory power that also allows for streaming computation, which is an important feature for sensor networks moving over a long period of time.

It turns out that no method relying on connectivity data alone can determine in all cases if an evasion path exists. Indeed, we provide examples showing that the fibrewise homotopy type of the sensor network does not determine the existence of an evasion path; the embedding of the sensor network in spacetime also matters. We therefore consider a stronger model for planar sensors which measure cyclic orderings and alpha complexes, and given this model we provide necessary and sufficient conditions for the existence of an evasion path.

We end with two possible directions for future research. First, we are interested in the open question from Section~\ref{S: Sensors Measuring Cyclic Orderings}: can one determine the existence of an evasion path using only \u Cech complexes and the cyclic ordering data? An answer to this question would fill the gap between Theorem~7 of \cite{Coordinate-free} (or equivalently our Theorem~\ref{T: longBarcodeNecessary}) and our Theorem~\ref{T: Stronger sensors}. Second, the evasion problem motivates a natural extension discussed in \cite{MyThesis}: can we describe the entire space of evasion paths? Knowledge about the space of evasion paths may be helpful in determining how to best patch a sensor network that contains an evasion path. Alternatively, we may want to find the evasion path that maintains the largest separation between the intruder and the sensors, that requires an intruder to move the shortest distance, or that requires an intruder to move at the lowest top speed. Knowledge about the space of sections may be helpful for such problems.

\subsection*{Funding}

This work was supported by the National Science Foundation [DMS 0905823, DMS 0964242]; the Air Force Office of Scientific Research [FA9550-09-1-643, FA9550-09-1-0531]; and the National Institutes of Health [I-U54-ca49145-01]. H.~Adams was supported by a Ric Weiland Graduate Fellowship at Stanford University.

\appendix
\section{Index to Multimedia Extensions}
\begin{tabular}{lll}
\hline
Extension & Media Type & Description \\
\hline
1 & Video & \href{http://www.ima.umn.edu/~henrya/research/DependenceOnTheEmbedding.avi}{Video of the two sensor networks in Figure~\ref{F: cartoonYesNo}} \\
\hline
\end{tabular}

\section{\u Cech Complex Approximations}\label{A: Cech Complex Approximations}

In this appendix we explain the Vietoris--Rips approximation to the \u Cech complex. We still assume that each sensor covers a ball of radius one, but we now consider different communication distances between the sensors. Two sensors no longer detect when they overlap, but instead when their centers are within communication distance $2\epsilon$. Hence the sensors measure a time-varying communication graph which at time $t$ has an edge between sensors $v$ and $\tilde{v}$ if $\|v(t) -\tilde{v}(t)\| \leq 2\epsilon$.

The Vietoris--Rips complex $\VR(t, \epsilon)$ is the maximal simplicial complex built on top of the connectivity graph with communication distance $2\epsilon$ at time $t$. Equivalently, a simplex is in $\VR(t, \epsilon)$ when its diameter is at most $2\epsilon$ \cite{Vietoris}. Note a simplex is included in $\VR(t, \epsilon)$ when all its edges are in the connectivity graph, and so the Vietoris--Rips complex can be constructed from the connectivity graph. See Figure~\ref{F: sensorBalls_Rips_Cech} for an example.

\begin{figure}[h]
\begin{center}
	\begin{subfigure}[t]{2.5in}
		\includegraphics[width=2.5in]{sensorBalls_Cech.pdf}
		\caption{\u Cech complex}
	\end{subfigure}
	\hspace{10mm}
	\begin{subfigure}[t]{2.5in}
		\includegraphics[width=2.5in]{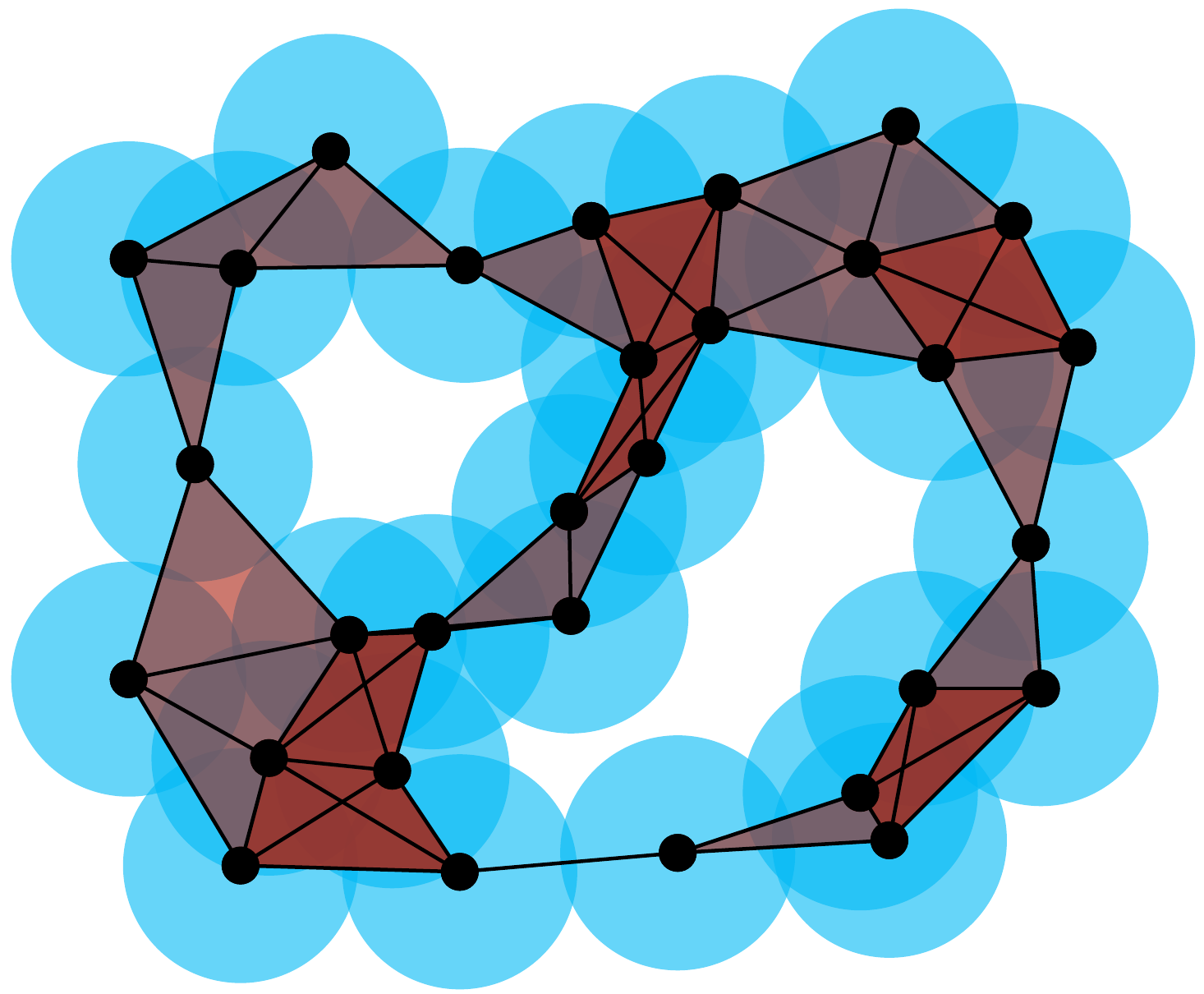}
		\caption{Vietoris--Rips complex for $\epsilon = 1$}
	\end{subfigure}
\end{center}
\caption{Note the 2-simplex that is absent from the \u Cech complex but is present in the Vietoris--Rips complex.}
\label{F: sensorBalls_Rips_Cech}
\end{figure}

By changing the communication distance of the sensors, we can approximate the \u Cech complex from either direction using a Vietoris--Rips complex. Jung's Theorem \cite{Jung} implies
$$\VR\Biggl(t, \sqrt{\frac{d+1}{2d}}\Biggr) \subset C(t) \subset \VR(t, 1).$$
Let $p \colon \VR(t, \epsilon) \to \R^d$ be the projection of the Vietoris--Rips complex into $\R^d$. It follows from Jung's Theorem that if every continuous map $s \colon I \to \xD$ satisfies $s(t) \in p\Bigl(\VR\Bigl(t, \sqrt{\frac{d+1}{2d}}\Bigr)\Bigr)$ for some $t$, then there is no evasion path. Similarly, if there is a continuous map $s \colon I \to \xD$ with $s(t) \notin p(\VR(t, 1))$ for all $t$, then there is an evasion path. Hence we can use the time-varying Vietoris--Rips complex to prove one-sided results about the existence of an evasion path. For example, when $d = 2$ the bound $\VR\Bigl(t, \sqrt{\frac{d+1}{2d}}\Bigr) \subset C(t)$ is closely related to Lemma~1 of \cite{Coordinate-free}, and is used to provide necessary conditions for the existence of an evasion path.

\bibliographystyle{alpha}
\bibliography{EvasionPathsInMobileSensorNetworks}

\end{document}